\include{BoxedEPS}

\documentclass[12pt]{amsart}
\usepackage{amsmath,amsthm}
\usepackage{graphicx,amsmath,amsfonts}

\chardef\bslash=`\\ 

\makeatletter
\def\verbatim{\interlinepenalty\@M \@verbatim
  \leftskip\@totalleftmargin\advance\leftskip2pc
  \frenchspacing\@vobeyspaces \@xverbatim}
\makeatother
\newtheorem{thm}{Theorem}[section]
\newtheorem{cor}[thm]{Corollary}
\newtheorem{lem}[thm]{Lemma}
\newtheorem{prop}[thm]{Proposition}

\newtheorem{rem}[thm]{Remark} 

\numberwithin{equation}{section}

\newcommand{\begeq}{\begin {equation}}

\newcommand{\bs}{\begin {split}}
\newcommand{\es}{\end {split}}
\newcommand{\bp}{\begin {prop}}
\newcommand{\ep}{\end {prop}}
\newcommand{\bt}{\begin {thm}}
\newcommand{\et}{\end {thm}}
\newcommand{\bc}{\begin {cor}}
\newcommand{\ec}{\end {cor}}
\newcommand{\bl}{\begin {lem}}
\newcommand{\el}{\end {lem}}
\newcommand{\bpf}{\begin {proof}}
\newcommand{\epf}{\end {proof}}
\newcommand{\bi}{\begin {itemize}}
\newcommand{\ei}{\end {itemize}}
\newcommand{\ben}{\begin {enumerate}}
\newcommand{\een}{\end {enumerate}}
\newcommand{\brem}{\begin {rem}}
\newcommand{\erem}{\end {rem}}

\newcommand{\ZZ}{{\mathbb Z}}
\newcommand{\CC}{{\mathbb C}}
\newcommand{\NN}{{\mathbb N}}

\newcommand{\Rd}{ {\Bbb R}^d}
\newcommand{\Zd} {{\Bbb Z}^d}

\begin{document}
\bibliographystyle{plain}

\title[Stability on Weighted $L^p$ spaces] {Stability of Localized Integral Operators on Weighted $L^p$ spaces
}

\author{Kyung Soo Rim,
Chang Eon Shin
 and  Qiyu Sun}

 \thanks{The first two authors are partially supported by Basic Science Research Program through the National Research Foundation of Korea (NRF) funded by the Ministry of Education, Science and Technology (2010-0028130).}

\address{Kyung Soo Rim: Department of Mathematics, Sogang University, Seoul,  Korea. Email: ksrim@sogang.ac.kr}

\address{Chang Eon Shin:  Department of Mathematics, Sogang University, Seoul,  Korea. Email: shinc@sogang.ac.kr }

\address{Qiyu Sun: Department of Mathematics,  University of Central Florida,
Orlando, FL 32816, USA. Email: qiyu.sun@ucf.edu}


\date{\today }

\subjclass{47G10,   45P05,  47B38,  31B10, 42C99, 44A35, 46E30}

\keywords{Integral operator,  weighted function space,  Muckenhoupt weight, spectrum,  Bessel potential, infinite matrix, Wiener's lemma, bootstrap technique, reverse H\"older inequality, doubling measure}


\maketitle
\begin{abstract}  In this paper, we consider localized   integral operators 
  whose kernels
 have mild singularity near the diagonal and certain H\"older regularity and decay off the diagonal.
 Our model example  is the Bessel potential operator ${\mathcal J}_\gamma, \gamma>0$.
 We show that if such a localized integral operator  has  stability on a weighted  function space $L^p_w$ for some
 $p\in [1, \infty)$ and Muckenhoupt $A_p$-weight $w$, then it has
   stability on weighted  function spaces $L^{p'}_{w'}$ for all $1\le p'<\infty$ and Muckenhoupt $A_{p'}$-weights $w'$.
\end{abstract}

\section{Introduction}

Let $K$ be  a kernel function on $\Rd\times \Rd$. Define the minimal radial function   on $\Rd$
that is radially decreasing and   dominates the off-diagonal decay of the kernel  $K$ by
\begin{equation}
\label{rk.def} r_K(x):=\sup_{|y-y'|\ge |x|}|K(y,y')|.
\end{equation}
Here $|x|:=\max\{|x_1|,\cdots, |x_d|\}$ for $x:=(x_1, \cdots, x_d)\in\mathbb{R}^d$.
In this paper, we consider integral operators
 \begin{equation}
\label{integraloperator.def}
Tf(x):=\int_{\Rd} K(x,y) f(y)  dy
\end{equation}
whose kernel $K$ on $\Rd\times \Rd$ has its off-diagonal decay dominated by an integrable radially decreasing function on $\Rd$, i.e.,
 \begin{equation}\label{rkintegrable.eq}
 \|r_K\|_1:=\int_{\Rd} r_K(x) dx<\infty.
 \end{equation}
The model example of such an integral operator  is the Bessel potential \cite{steinbook70}
\begin{equation}
{\mathcal J}_\gamma f=\int_{\Rd}  G_\gamma(x-y) f(y) dy,\ \gamma>0,
\end{equation}
where the Bessel kernel $G_\gamma$ is defined with the help of Fourier transform by
$$\widehat G_\gamma(\xi_1, \ldots, \xi_d)=\big(1+|\xi_1|^2+\cdots+|\xi_d|^2)^{-\gamma/2},$$
and  the Fourier transform $\hat f$ of an integrable function $f$ is defined by
$\hat f(\xi)=\int_{\Rd} f(x) e^{-ix\xi} dx$.

\smallskip

For $1\le p<\infty$, we say that  a weight  $w$ on the $d$-dimensional Euclidean space $\Rd$ (i.e., a positive locally-integrable  function on $\Rd$)  is an
{\em $A_p$-weight}  if
\begin{equation}\label{apweight.eq1}
\Big(\frac{1}{|Q|} \int_Q w(x) dx\Big) \Big(\frac{1}{|Q|} \int_{Q} w(x)^{-\frac{1}{p-1}} dx\Big)^{p-1} \le A<\infty\quad {\rm for \ all\ cubes} \ Q\end{equation}
when $1<p<\infty$, and if
\begin{equation}
\label{apweight.eq2}\frac{1}{|Q|} \int_Q w(y) dy \le A \inf_{x\in Q} w(x)\quad {\rm for \ all\ cubes} \ Q \end{equation}
when $p=1$ \cite{fourieranalysisbook, rubiobook, steinbook93}. Here $|E|$ stands for  the Lebesgue measure of a measurable set $E\subset \Rd$.
 The {\em $A_p$-bound} of an $A_p$-weight $w$, to be denoted by $A_p(w)$,
    is the  smallest constant $A$ for which \eqref{apweight.eq1}  holds when $1<p<\infty$ (respectively \eqref{apweight.eq2}  holds when $p=1$).
Simple nontrivial example of $A_p$-weights is the polynomial weight   $w_\alpha(x):=|x|^\alpha$, which  is an $A_p$-weight if the exponent $\alpha$ of
 the polynomial weight $w_\alpha$ satisfies $-d<\alpha\le 0$ for $p=1$, and  $-d<\alpha<d(p-1)$ for $1<p<\infty$.

\smallskip

Denote by $I$ the identity operator, and by $L^p_w:=L^p_w(\Rd)$  the space of all measurable  functions  $f$ on $\Rd$ with   $\|f\|_{p, w}:=(\int_{\Rd} |f(x)|^p w(x) dx)^{1/p}<\infty$.
 A well-known result about the integral operator $T$ in  \eqref{integraloperator.def}
 is that it is bounded on  the weighted function space $L^p_w$ for any $p\in [1, \infty)$ and  $A_p$-weight $w$.
Furthermore  there exists an absolute  constant
 $C$, that depends on $p$ and $d$ only, such that
 \begin{equation} \label{lioboundedness.def}
\|Tf\|_{p,w} \le C (A_p(w))^{1/p} \|r_K\|_1
\|f\|_{p,w}
\end{equation}
for all $A_p$-weights $w$ and functions $f\in L^p_w$,
see also Proposition \ref{lioboundedness.prop1}.
In this paper, instead of  establishing  boundedness of the integral operator $T$ on $L^p_w$,
 we   consider stability  of integral operators $zI-T, z\in \CC$, on $L^p_w$, i.e.,
 there exists a positive constant $C$ such that
\begin{equation}
\|(zI-T)f\|_{p,w}\ge C \|f\|_{p,w} \ {\rm for \ all} \ f\in L^p_w.
\end{equation}
We will show that the stability of integral operators $zI-T, z\in \CC$, on $L^p_w$
for different $p\in [1, \infty)$ and $A_p$-weights $w$ are equivalent to each other, provided that the kernel $K$
of the integral operator $T$ is  assumed, in addition to its off-diagonal decay dominated by an integrable radially decreasing function,
to have certain H\"older regularity off the diagonal and mild singularity near the diagonal, i.e.,
  \begin{equation}
\label{kernelcondition} 
\|r_K\|_1
+\sup_{0<\delta\le 1}  \delta^{-\alpha} \|r_{\omega_\delta(K)}\|_1+ \sup_{0<\delta\le 1} \delta^{-\alpha} \|r_K \chi_{|\cdot|\le \delta}\|_1<\infty
 \end{equation}
for some $\alpha\in (0,1]$. Here for a kernel function $K$ on $\Rd\times \Rd$, its modified modulus of continuity $\omega_\delta(K)$  is defined by
\begin{equation}
\label{lio.eq2} \omega_\delta(K)(x,y)=\left\{\begin{array}{ll}
\sup_{|x'-x|, |y'-y|\le \delta} |K(x', y')-K(x,y)|  & {\rm if} \ |x-y|\ge 4\delta,\\
0 & {\rm otherwise}.\end{array}\right.
\end{equation}

\begin{thm}\label{liostability.tm} Let $z\in \CC$,  $K$ be a kernel function on $\Rd\times \Rd$ satisfying \eqref{kernelcondition} for some $\alpha\in (0,1]$, and let
$T$ be the integral operator in \eqref{integraloperator.def} with kernel $K$.
If $z I-T$ has stability on $L^p_w$ for some $1\le p<\infty$ and  $A_p$-weight $w$, then
it has  stability on  $L^{p'}_{w'}$ for all $1\le p'<\infty$ and  $A_{p'}$-weights $w'$.
\end{thm}

 Denote by $s_{p,w}(T)$  the set of all complex numbers $z$ such that $zI-T$ does not have stability on $L^p_w$, and by $s_p(T)$ instead of $s_{p, w_0}(T)$ for short when  $w$ is the  trivial weight $w_0\equiv 1$.
Then  Theorem \ref{liostability.tm} can be reformulated  as follows:
\begin{equation}\label{lioresiduespectrum.eq}
s_{p,w}(T)=s_{2}(T)
\end{equation}
 for all $1\le p< \infty$ and $A_p$-weights $w$, provided that the kernel of the integral operator $T$ satisfies \eqref{kernelcondition}.
We remark that for the  operator $zI-T$, it is established in \cite{shincjfa09} the equivalence of its stability on unweighted function spaces $L^p$
for different exponents $p\in [1, \infty]$, i.e.,
 \begin{equation}s_p(T)=s_{2}(T)\end{equation}
  for all $1\le p\le \infty$. The assumption on the kernel $K$ of the operator $T$ in  the above equivalence
is that it has certain  H\"older regularity and  its off-diagonal decay dominated by a function in the Wiener amalgam space ${\mathcal W}_1$,
 the space containing all measurable functions $h$ on $\Rd$ with 
$\| h\|_{\mathcal{W}_1}:=\sum_{k\in\mathbb{Z}^d} \sup_{x\in [-1/2,1/2)^d}|h(k+x)|<\infty$.
More precisely, the kernel $K$ satisfies the following condition:
\begin{equation}\label{liostability.rem.eq1}
\Big\| \sup_{y\in\mathbb{R}^d} |K(y,\cdot+y)|\Big \|_{\mathcal{W}_1} +\sup_{0<\delta\le 1} \delta^{-\alpha} \Big\| \sup_{y\in\mathbb{R}^d}\tilde \omega_\delta(K)(y,\cdot +y)\Big\|_{\mathcal{W}_1}<\infty
\end{equation}
for some $\alpha\in (0,1]$, where
the module of continuity $\tilde \omega_\delta(K), \delta>0$, of a kernel $K$ on $\Rd\times \Rd$ is defined by
$$
\tilde \omega_\delta( K)(x, y)=\sup_{\max(|x'-x|, |y'-y|)\le\delta}|K(x', y')-K(x,y)|\quad {\rm for \ all} \ x, y\in \Rd
$$
c.f. the modified module  of continuity $\omega_\delta(K)$ of a kernel $K$ in \eqref{lio.eq2}.
The assumptions  \eqref{kernelcondition} and  \eqref{liostability.rem.eq1}  on kernels  are not comparable. Kernels satisfying
\eqref{kernelcondition} could have certain blowup  near the diagonal
 while  kernels satisfying \eqref{liostability.rem.eq1} do not allow any singularity (and even require certain regularity) near the diagonal.
On the other hand,
 kernels satisfying \eqref{liostability.rem.eq1}
 have less requirement on the  decay far away from the diagonal than kernels satisfying \eqref{kernelcondition} do.

 \smallskip
 We say that an integral operator $T$ in \eqref{integraloperator.def} is {\em  of convolution type} (or a {\em convolution operator}) if its kernel $K$
 can be written as $K(x,y)=g(x-y)$ for some integrable function $g$ on $\Rd$ \cite{barnes90, hulanicki, kurbatovbook}. In this case, one may verify that
 $s_2(T)=\{\hat g(\xi)| \ \xi\in \Rd\}\cup \{0\}$. This together with \eqref{lioresiduespectrum.eq}
 implies that
 $$ s_{p,w} (T)= \{\hat g(\xi)| \ \xi\in \Rd\}\cup \{0\}$$
 for all $1\le p<\infty$ and $A_p$-weight $w$,  provided that
 $Tf(x)=\int_{\Rd} g(x-y) f(y) dy$ for some integrable function $g$ on $\Rd$ and the kernel  $g(x-y)$ satisfies \eqref{kernelcondition}.
 Thus for the Bessel potentials ${\mathcal J}_\gamma, \gamma>0$, we have that
 $ s_{p,w} ({\mathcal J}_\gamma)= [0,1]$
 for all $1\le p<\infty$ and $A_p$-weight $w$, which is new up to our knowledge.

 \smallskip

 Denote by $\sigma_{p,w}(T)$  the spectrum of the operator $T$ on $L^p_w$ and by $\sigma_p(T)$ instead of $\sigma_{p, w_0}(T)$ for short when  $w$ is the  trivial weight $w_0\equiv 1$.
   Clearly we have that
 \begin{equation}s_{p,w}(T)\subset \sigma_{p,w}(T)
 \end{equation}
 for all bounded operators $T$ on $L^p_w$.  We are working on the problem whether or not the above inclusion
 is indeed an equality when the kernel of the operator $T$ satisfies \eqref{kernelcondition}.
 The reader may refer to
\cite{barnes90, baskakov11, brandenburg75, farrell10,  hulanicki, kurbatovbook, pytlik, shincjfa09}
 for spectra $\sigma_{p,w}(T)$ of  various integral operators $T$, and \cite{grochenigsurvey, suntams07, sunconst10, sunacha08}
 for its connection to Wiener's lemma for infinite matrices.

\smallskip

The paper is organized as follows. In Section \ref{preliminary.section}, we provide some preliminary results on the boundedness, approximation and discretization of the integral operator $T$  in \eqref{integraloperator.def} on weighted function spaces $L^p_w$, and also the boundedness on weighted sequence spaces and off-diagonal decay for the discretization  of the integral operator $T$  in \eqref{integraloperator.def} at different levels.  The main result of this paper is Theorem \ref{liostability.tm}, whose proof is given in  Section \ref{liostability.section}.
Some refinements of doubling  measure property and reverse H\"older inequality for  Muckenhoupt $A_p$-weights are
included in the appendix.

\smallskip

In this paper, we will use the following notation.
${\mathbb Z}_+:= {\mathbb N}\cup\{0\}$;  $\ell^p_w:=\ell^p_w(\Lambda)$ is
the space of all  weighted $p$-summable column vectors $c=(c(\lambda))_{\lambda\in \Lambda}$ with
 $\|c\|_{p,w}:=(\sum_{\lambda\in \Lambda} |c(\lambda)|^p w(\lambda))^{1/p}<\infty$, where $1\le p<\infty$ and  $w=(w(\lambda))_{\lambda\in \Lambda}$ is a weight; $\langle g_1, g_2\rangle:=\int_{\Rd} g_1(x) \overline{g_2(x)} dx$  provided that $g_1 g_2$ is integrable;
${\mathcal A}_p, 1\le p<\infty$, is the set of all $A_p$-weights;
 $kQ$ stands for the cube with the
 center same as the one of the  given cube $Q$ and   the radius $k$ times the one of cube $Q$;  $b_K$ is the function on the positive axis such that $b_K(|x|)=r_K(x)$ is the minimally radically decreasing function in \eqref{rk.def} that dominates the off-diagonal decay of a kernel $K$ on $\Rd\times \Rd$; and
 $C$ denotes an absolute constant which could be different at different occurrences.

\section{Preliminary}\label{preliminary.section}

We divide this section  into two parts.
In the first part of this section, we consider the boundedness, approximation and discretization of
an integral operator
 whose kernel has certain off-diagonal decay and H\"older regularity.
 In the first subsection we recall that
an integral operator, whose  kernel  has its off-diagonal decay  dominated by an integrable radially decreasing function,
is a bounded operator on $L^p_w$ for any $1\le p<\infty$ and $A_p$-weight $w$, see Proposition \ref{lioboundedness.prop1}.
Define   $P_n, n\in {\mathbb Z}$, on $L^p_w$ by 
\begin{equation}\label{pn.def}
P_nf =\sum_{\lambda\in 2^{-n} \Zd} \langle f, \phi_{n,2^n \lambda}\rangle \phi_{n,2^n\lambda},\
f\in L^p_w,
\end{equation}
where  $\phi_{n,k}=2^{nd/2} \chi_{[-1/2, 1/2)^d}(2^n \cdot-k),\ n\in \ZZ, k\in \Zd$.
For $p=2$ and the trivial weight $w\equiv 1$, $P_n, n\in \ZZ$, are  projection operators onto  $V_n:=P_nL^2$, which form a multiresolution analysis associated with the Haar  wavelet system
\cite{daubecheiesbook}. In the second subsection, we prove that  an integral operator $T$  with its kernel
 having certain off-diagonal decay, mild singularity near the diagonal  and  H\"older regularity can be approximated by $P_nT, TP_n$ and $P_nTP_n, n\in \ZZ$, in the operator norm on $L^p_w$, see Proposition \ref{lioapproximation.prop2}. As a consequence of the above approximation, we conclude that zero is  in the spectrum of a localized integral operator, see Corollary \ref{zerospectrum.cor}.
We call the operator $P_nTP_n$  the {\em discretization of the integral operator $T$ at $n$-th level}, as they are closely related to  infinite matrices
 \begin{equation}\label{an.def}
A_n:=\big(a_n(\lambda,\lambda')\big)_{\lambda, \lambda'\in 2^{-n}\Zd},\ n\in \ZZ\end{equation}
where
\begin{equation*}
a_n(\lambda, \lambda')=2^{nd} \int_{\Rd} \int_{\Rd} \phi_{n, 2^n\lambda}(x) K(x,y) \phi_{n, 2^n\lambda'}(y) dydx, \ \lambda, \lambda'\in 2^{-n} \Zd,\end{equation*}
see Proposition \ref{liodiscretization.prop1} of the third subsection. The same discretization has been used in \cite{shincjfa09, sunacha08}
to establish Wiener's lemma  and stability for localized integral operators on unweighted function spaces $L^p, 1\le p<\infty$.

In the second part of this section, we consider the boundedness  and off-diagonal decay property of   discretization matrices $A_n, n\in \ZZ$.
 Given a locally integrable positive function $w$, define its {\em discretization at $n$-th level} by
\begin{equation}\label{wnlambda.def} w_n:=(w_n(\lambda))_{\lambda\in 2^{-n}\Zd},
\end{equation} where
$w_n(\lambda)=2^{nd} \int_{\lambda+2^{-n}[-1/2, 1/2)^d} w(x) dx, \lambda\in 2^{-n}\Zd$.
As shown in Proposition \ref{discreteapweight.prop1},  discretization   of an $A_p$-weight $w$  at any  level
is  a discrete $A_p$-weight, see \eqref{discreteaqweight.eq1} and \eqref{discreteaqweight.eq2} for the definition.
 In
Proposition \ref{anboundedness.prop} of the fourth subsection, we show  that for every $n\in \ZZ$, the discretization matrix $A_n$ is bounded on  the weighted sequence space $\ell^p_{w_n}$
for any $1\le p<\infty$ and $A_p$-weight $w$.  The above proposition can be thought as a discretized version of Proposition \ref{lioboundedness.prop1}.
As we always assume in the paper that  the integral operator $T$ in \eqref{integraloperator.def}
has its kernel with  certain off-diagonal decay, 
its discretization matrices $A_n, n\in \ZZ$, have similar off-diagonal decay, see Proposition \ref{offdiagonaldecay.prop} of the fifth subsection.
For $N\ge 1$ and $k\in N\Zd$, define the  localization matrix   $\Psi_k^N$ on a sequence space on $2^{-n}\Zd$
by
\begin{equation}\label{lio.eq6}
(\Psi_k^Nc)(\lambda):= \psi_0((\lambda-k)/N) c(\lambda)\quad {\rm for} \  c:=(c(\lambda))_{\lambda\in 2^{-n}\Zd},
\end{equation}
 where
$\psi_0(x)=\max(\min(2-|x|, 1), 0)$.
In the sixth subsection, we  prove that the commutators $[A_n, \Psi_k^N]:=A_n \Psi_k^N-\Psi_k^N A_n$  between the  discretization matrices $A_n$ and the localization
matrices $\Psi_k^N$ have certain off-diagonal decay, see  Proposition \ref{liodiscretization.prop3}. The above off-diagonal decay property for the commutators $[A_n, \Psi_k^N]$ plays crucial roles in the proof of Theorem \ref{liostability.tm}. We remark that similar off-diagonal decay property for the
commutator $[A_n, \Psi_k^N]$ has been used in \cite{shincjfa09}  to establish the  equivalence of stability of a localized
 integral operator  on unweighted function space $L^p$ for different exponent $1\le p<\infty$.

\subsection{Boundedness of  localized integral operators} 

\begin{prop}
\label{lioboundedness.prop1}
Let $1\le p<\infty$ and $K$ be a kernel function on $\Rd\times \Rd$ whose off-diagonal decay is dominated by an integrable radially decreasing function
(i.e.,
\eqref{rkintegrable.eq} holds).
Then   the integral operator  $T$ in \eqref{integraloperator.def} with kernel $K$
 is a bounded operator on   $L_w^p$  for any $A_p$-weight $w$.
Furthermore,
\begin{equation} \label{lioboundedness.prop1.eq3}
\|Tf\|_{p,w} \le C (A_p(w))^{1/p} \|r_K\|_1
\|f\|_{p,w}
\end{equation}
for all weights $w\in {\mathcal A}_p$ and functions $f\in L^p_w$, where $C$ is an absolute  constant that depends on $p$ and $d$ only.
\end{prop}

\begin{proof} It is well known that the integral operator $T$ in \eqref{integraloperator.def} is a bounded operator on $L^p_w$
 \cite{fourieranalysisbook, rubiobook,   steinbook93}. 
  We include a sketch of the proof for the bound estimate in \eqref{lioboundedness.prop1.eq3} 
    and for the completeness of the paper.
Note that
\begin{equation}\label{lioboundedness.prop1.pf.eq1}
|Tf(x)| 
 \le  \sum_{j\in \ZZ} b_K(2^{j-1}) \int_{2^{j-1}\le |x-y|<2^j} |f(y)| dy \quad {\rm for \ all} \ x\in \Rd
\end{equation}
(and hence $|Tf(x)|$ is dominated by a constant multiple of the maximal function $Mf(x)$, which is bounded on $L^p_w$ for all $1<p<\infty$ and $A_p$-weights $w$ \cite{fourieranalysisbook, rubiobook,  steinbook93}).
Then for $p=1$,
\begin{eqnarray*}
\|Tf\|_{1,w} & \le & \sum_{j\in \ZZ} b_K(2^{j-1})\int_{\Rd} w(x) \int_{2^{j-1}\le |x-y|<2^j} |f(y)|dy dx
\nonumber
\\
& \le & A_1(w) \Big(\sum_{j\in \ZZ} b_K(2^{j-1}) 2^{(j+1)d}\Big) \|f\|_{1,w} 
\le C A_1(w) \|r_K\|_1 \|f\|_{1,w}.
\label{eq.TF}
\end{eqnarray*}
This proves  \eqref{lioboundedness.prop1.eq3} for $p=1$.

For $1<p<\infty$, applying \eqref{lioboundedness.prop1.pf.eq1} and  using H\"older inequality, we obtain
\begin{eqnarray*}
|Tf(x)|^p 
& \le &
C  A_p(w) \|r_K\|_1^{p-1} \sum_{j\in \ZZ}
 b_K(2^{j-1}) 2^{jd} \Big( \int_{|x-y'|<2^{j}}  w(y')dy'\Big)^{-1}
 \nonumber\\
 & &  \quad \times \Big(\int_{2^{j-1}\le |x-y|<2^{j}} |f(y)|^p w(y) dy\Big).
\end{eqnarray*}
Thus
\begin{eqnarray*} 
\|Tf\|_{p,w}^p 
 & \le &
 C A_p(w) \|r_K\|_1^{p-1}
 \sum_{j\in \ZZ} b_K(2^{j-1}) 2^{jd}\nonumber\\
& &\times \int_{\Rd} |f(y)|^p w(y)  \Big(\int_{2^{j-1}\le |x-y|<2^{j}}
  \frac{w(x)} {\int_{|x-y'|<2^{j}}  w(y')dy'}dx\Big) dy \nonumber\\
 &\le &   C A_p(w) \|r_K\|_1^{p-1}
 \sum_{j\in \ZZ} b_K(2^{j-1}) 2^{jd}
 \int_{\Rd} |f(y)|^p w(y) \nonumber\\
 & & \times \Big(
 \sum_{\epsilon\in \{-1, 0, 1\}^d} \int_{|x-y-\epsilon 2^{j-1}|<2^{j-1}}
  \frac{w(x)} {\int_{|y'-y-\epsilon 2^{j-1} |<2^{j-1}}  w(y')dy'}dx\Big) dy \nonumber\\
 & \le & C A_p(w) \|r_K\|_1^p \|f\|_{p,w}^p.
\end{eqnarray*}
This establishes \eqref{lioboundedness.prop1.eq3} for $1<p<\infty$ and completes the proof.
\end{proof}


\subsection{Approximation of  localized integral operators} 

\begin{prop}\label{lioapproximation.prop2}
 Let $1 \le p <\infty$,    $w$ be an $A_p$-weight,  $K$ be a kernel function on $\Rd\times \Rd$ satisfying
 \eqref{kernelcondition} for some $\alpha\in (0, 1]$,
 $T$  be the integral operator in \eqref{integraloperator.def} with kernel $K$, and let $P_n, n\in \ZZ$, be as in \eqref{pn.def}.
 Then there exists an absolute constant $C$ (depending on $p$ and $d$ only)  such that
 \begin{eqnarray}
 \label{lioapproximation.prop2.eq5}
 & & \!\! \|(TP_n-T)f\|_{p,w}+\|(P_nT-T)f\|_{p,w}+\|(P_nTP_n-T)f\|_{p,w}\nonumber\\
 \quad \quad & \le &\!\!  C D_0 2^{-n\alpha} (A_p(w))^{1/p} \|f\|_{p,w} \quad {\rm for \ all} \  n\in \ZZ_+, w\in {\mathcal A}_p\ {\rm and}\ f\in L^p_w,
 \end{eqnarray}
 where $D_0=\|r_K\|_1 + \sup_{0<\delta\le 1} \delta^{-\alpha} \|r_K \chi_{|\cdot|\le \delta}\|_1
+\sup_{0<\delta\le 1}  \delta^{-\alpha} \|r_{\omega_\delta(K)}\|_1$.
 \end{prop}

We remark that it is established in \cite[Proof of Theorem 4.1]{shincjfa09} that a localized integral operator has
 the above approximation property   on  unweighted function spaces $L^p, 1\le p<\infty$.
By Proposition \ref{lioapproximation.prop2},  we see that  $TP_n, P_nT, P_nTP_n$ approximate the localized integral operator $T$ in the operator norm
$\|\cdot\|_{{\mathcal B}(L^p_w)}$ on $L^p_w$, as $n$ tends to infinity, i.e.,
\begin{equation}\label{lioapproximation.prop2.limiteq}
\lim_{n\to \infty} \|P_nT-T\|_{{\mathcal B}(L^p_w)}+  \|TP_n-T\|_{{\mathcal B}(L^p_w)}+ \|P_nTP_n-T\|_{{\mathcal B}(L^p_w)}=0.\end{equation}
As a consequence of the above limit,  zero is in the spectrum of a localized integral operator $T$ on $L^p_w$, c.f.  \cite[Theorem 2.2 (iv)]{sunacha08}.

\begin{cor}  \label{zerospectrum.cor}
Let the integral operator  $T$ be as in Proposition \ref{lioapproximation.prop2}. Then
$0\in s_{p,w}(T)\subset \sigma_{p,w}(T)$ for all $1\le p<\infty$ and $w\in {\mathcal A}_p$.
\end{cor}

\begin{proof}
Let $\varphi_0=\max(1-|x|, 0)$ be the hat function and set  $g_n:=\varphi_0-P_n \varphi_0, n\ge 0$. Note that $0\ne g_n\in L^p_w$ and  $P_n^2=P_n$ for all $n\in \ZZ_+$. Then for all $1\le p<\infty$ and $w\in {\mathcal A}_p$, we have that
\begin{equation}\label{zerospectrum.cor.pf.eq1}
 \inf_{\|g\|_{p,w}\ne 0} \frac{\|Tg\|_{p,w}}{\|g\|_{p,w}} \le   \frac{\|Tg_n\|_{p,w}}{\|g_n\|_{p,w}}=\frac{\|(T-TP_n)g_n\|_{p,w}}{\|g_n\|_{p,w}}
 \le   \|TP_n-T\|_{{\mathcal B}(L^p_w)} \to 0 \end{equation}
as $n\to \infty$
by  \eqref{lioapproximation.prop2.limiteq}. This proves the conclusion that $0\in s_{p,w}(T)\subset \sigma_{p,w}(T)$.
\end{proof}

Now we prove Proposition  \ref{lioapproximation.prop2}.

\begin{proof}[Proof of Proposition \ref{lioapproximation.prop2}] By \eqref{pn.def},  $P_n$ is an integral operator with kernel
\begin{equation*} 
 P_n(x,y):=\left\{\begin{array}{ll} 2^{nd}   & {\rm if} \ x, y\in 2^{-n}(k+[-1/2, 1/2)^d)\
    {\rm for \ some} \ k\in \Zd,
\\ 0 & {\rm otherwise},\end{array}\right.
\end{equation*}
and
 \begin{eqnarray} \label{lioapproximation.prop.pf.eq1}
\|P_n f\|_{p,w}^p & = & 2^{ndp/2} \sum_{\lambda\in 2^{-n} \Zd} |\langle f, \phi_{n, 2^{n}\lambda}\rangle|^p \int_{\lambda+2^{-n}[-1/2,1/2)^d} w(x)dx\nonumber\\
& \le & A_p(w) \sum_{\lambda\in 2^{-n} \Zd} \int_{\lambda+2^{-n}[-1/2,1/2]^d} |f(x)|^p w(x) dx= A_p(w) \|f\|_{p,w}^p
\end{eqnarray}
for all $f\in L^p_w$.
Thus  $TP_n-T,P_nT-T$ and
 $P_nTP_n-T$ are bounded operators on $L^p_w$ by \eqref{lioapproximation.prop.pf.eq1} and  Proposition \ref{lioboundedness.prop1}.

 Denote by $K_n(x,y)$ the kernel of the integral operator $P_nTP_n-T$.
Then
\begin{eqnarray*} 
|K_n(x,y)|& = & \Big|\int_{\Rd} \int_{\Rd} (K(x',y')-K(x,y)) P_n(x, x') P_n(y',y) dx'dy'\Big|\nonumber\\
& \le & 2^{2nd} \int_{|x'-x|\le 2^{-n}} \int_{|y'-y|\le 2^{-n}} |K(x',y')-K(x,y)| dx'dy'
\nonumber\\
& \le & 
 2^{2d} r_{\omega_{2^{-n}}(K)}(x-y)
\end{eqnarray*}
for all $x, y\in \Rd$ with $|x-y|> 6\cdot2^{-n}$,
and
\begin{eqnarray*}
|K_n(x,y)| & \le &  |K(x,y)|+ \int_{\Rd}\int_{\Rd} \big|P_n(x, x')K(x',y')P_n(y',y)\big| dx'dy'\nonumber\\
& \le & r_K(x-y)+ 2^{nd} \int_{|t|\le 8\cdot 2^{-n}} r_K(t) dt
\end{eqnarray*}
for all $x, y\in \Rd$ with $|x-y|\le 6\cdot 2^{-n}$. Thus the kernel  $K_n(x,y)$ of the integral operator $P_nTP_n-T$ is dominated by
 $h_n(x-y)$, where $h_n$ is a radially decreasing function defined by
$$h_n(x):= \begin{cases}
r_K(x)+ 2^{nd}\int_{|t|\le 8\cdot 2^{-n}} r_K(t) dt & {\rm if} \ |x|\le 6\cdot 2^{-n},\\
2^{2d} r_{\omega_{2^{-n}}(K)}(x)  & {\rm if} \ |x|>6\cdot 2^{-n}.
\end{cases} $$
Similarly we can show that  kernels of the integral operators $TP_n-T$ and $P_nT-T$ have their off-diagonal decay dominated by  the same radially decreasing function $h_n$.
Then the desired estimate \eqref{lioapproximation.prop2.eq5} for the integral operators $TP_n-T, P_nT-T$ and $P_nTP_n-T, n\in \ZZ_+$,
follows from \eqref{kernelcondition}, Proposition \ref{lioboundedness.prop1} and the above observation about their kernels.
\end{proof}


\begin{rem}\label{lioapproximation.prop1} {\rm Let $1\le p< \infty$, $w$ be an $A_p$-weight, and $P_n, n\in \ZZ$, be as in
\eqref{pn.def}. For $n\in \ZZ$, define
 \begin{equation*}\label{lioapproximation.prop1.eq-1}
V^n_{p,w}=\Big\{ \sum_{\lambda\in 2^{-n} \Zd} c(\lambda)\phi_{n, 2^n\lambda}\
\Big| \  \big(c(\lambda)\big)_{\lambda\in \in 2^{-n}\Zd} \in \ell^p_{w_n}\Big\}.\end{equation*}
Then it follows from \eqref{pn.def}  and \eqref{lioapproximation.prop.pf.eq1} that
$P_n, n\in \ZZ$, are   bounded operators from $L^p_w$ onto $V^n_{p,w}\subset L^p_w$ with their operator norm bounded by $(A_p(w))^{1/p}$; i.e.,
$ V^n_{p,w}=P_n L^p_w$ and $\|P_n f\|_{p,w}\le (A_p(w))^{1/p} \|f\|_{p,w}$ for all $f\in L^p_w$.
} \end{rem}

\subsection{Discretization of localized integral operators and discretization matrices}

\begin{prop}\label{liodiscretization.prop1} Let $1\le p<\infty$, $w$ be an $A_p$-weight,
$K$ be a kernel function on $\Rd\times \Rd$ satisfying \eqref{rkintegrable.eq},
$T$ be the integral operator \eqref{integraloperator.def} with kernel $K$,  and let $P_n$ and $A_n , n\in \ZZ$, be as in \eqref{pn.def} and \eqref{an.def} respectively.  Then  
\begin{equation}\label{liodiscretization.prop1.eq1}
d_n(f)= 2^{-nd} A_n c_n(f) \quad {\rm for \ all} \ f\in L^p_w,
\end{equation}
where
$d_n(f)=\big(\langle P_n T P_n f, \phi_{n, 2^n\lambda}\rangle\big)_{\lambda\in 2^{-n}\Zd}$ and
$c_n(f)=\big(\langle P_n f, \phi_{n, 2^n\lambda}\rangle\big)_{\lambda\in 2^{-n}\Zd}$ for $f\in L^p_w$.
\end{prop}

\begin{proof} We mimic the argument in \cite[Proof of Theorem 4.1]{shincjfa09}. Note that
\begin{eqnarray*}
P_nTP_nf(x) & =&
\int_{\Rd} \int_{\Rd} \Big(\sum_{\lambda\in 2^{-n} \Zd} \phi_{n, 2^n\lambda}(x) \phi_{n, 2^n\lambda}(x')\Big)\nonumber\\
 & & \times K(x',y') \Big(\sum_{\lambda'\in 2^{-n}\Zd} \langle P_n f, \phi_{n, 2^n\lambda'}\rangle \phi_{n, 2^{n}\lambda'}(y') \Big) dx'dy'\nonumber\\
& = &  \sum_{\lambda\in 2^{-n}\Zd} \Big(2^{-nd} \sum_{\lambda'\in 2^{-n}\Zd}  a_n(\lambda, \lambda') \langle P_n f, \phi_{n, 2^n\lambda'}\rangle\Big) \phi_{n, 2^n\lambda}(x)
\end{eqnarray*}
for all $f\in L^p_w$.
Then \eqref{liodiscretization.prop1.eq1} follows.
\end{proof}

\subsection{Boundedness of discretization matrices}
\begin{prop}\label{anboundedness.prop} Let $1\le p<\infty$, $w$ be an $A_p$-weight,
$K$ be a kernel function on $\Rd\times \Rd$ satisfying 
 \eqref{rkintegrable.eq},
 and let $A_n=\big(a_n(\lambda, \lambda')\big)_{\lambda, \lambda'\in 2^{-n}\Zd}$ and $w_n , n\in \ZZ$, be as in  \eqref{an.def} and \eqref{wnlambda.def} respectively.
 Then  $A_n, n\in \ZZ$, are bounded operators on $\ell^p_{w_n}$ with operator norm bounded by
a constant multiple of
$2^{nd} (A_p(w))^{3/p}\|r_K\|_1$, i.e.,
\begin{equation}\label{anboundedness.prop.eq1} \|A_n c_n\|_{p, w_n}\le  C  2^{nd} \big(A_p(w)\big)^{3/p} \|r_K\|_1\|c_n\|_{p, w_n}\quad {\rm for \ all} \ c_n\in \ell^p_{w_n},
\end{equation}
where $C$ is an absolute constant depending on $p$ and $d$ only.
\end{prop}

\begin{proof}
Take $c_n:=(c_n(\lambda))_{\lambda\in 2^{-n}\Zd}\in \ell^p_{w_n}$ and set $f_n=\sum_{\lambda\in 2^{-n}\Zd} c_n(\lambda) \phi_{n, 2^n\lambda}$.
Then $$P_nTP_nf_n(x) =\sum_{\lambda\in 2^{-n}\Zd} \Big(2^{-nd} \sum_{\lambda'\in 2^{-n}\Zd}  a_n(\lambda, \lambda') c_n(\lambda')\Big) \phi_{n, 2^n\lambda}(x). $$
 This, together with
Proposition \ref{lioboundedness.prop1}, implies that
\begin{eqnarray*}
\|A_n c_n\|_{p, w_n} & = & 2^{nd(p+2)/(2p)} \|P_nTP_n f_n\|_{p,w}\le C  2^{nd(2+p)/(2p)} (A_p(w))^{3/p}  \|r_K\|_1 \|f_n\|_{p,w}\\
& =&
C 2^{nd}  (A_p(w))^{3/p}  \|r_K\|_1 \|c_n\|_{p,w_n},
\end{eqnarray*}
and hence  completes the proof. 
\end{proof}

\subsection{Off-diagonal decay  property of discretization matrices}
\begin{prop}\label{offdiagonaldecay.prop} Let $1\le p<\infty$, $w$ be an $A_p$-weight,
$K$ be a kernel function on $\Rd\times \Rd$ satisfying  
 \eqref{rkintegrable.eq},
 and let $A_n=\big(a_n(\lambda, \lambda')\big)_{\lambda, \lambda'\in 2^{-n}\Zd} , n\in \ZZ$, be as in  \eqref{an.def}.
 Then \begin{equation}
\label{offdiagonaldecay.prop.eq1}
|a_n(\lambda, \lambda')|\le  \left\{\begin{array} {ll} 2^{nd} \int_{|t|\le 3\cdot 2^{-n}} r_K(t) dt &
{\rm if} \ |\lambda-\lambda'|\le 2^{-n+1},\\
 r_K ((\lambda-\lambda')/2) & {\rm if}\ |\lambda-\lambda'|>2^{-n+1}.\end{array}\right.
\end{equation}
\end{prop}

\begin{proof} By  \eqref{an.def}, we obtain that
\begin{eqnarray*}
|a_n(\lambda, \lambda')| & \le &  2^{2nd}
\int_{|x-\lambda|\le 2^{-n-1}, |y-\lambda'|\le 2^{-n-1}} |K(x, y)| dydx
 \\
& \le & 2^{2nd} \int_{|x-\lambda|\le 2^{-n-1}}
\Big(\int_{|y-x|\le 3\cdot 2^{-n}} r_K(x-y) dy\Big) dx \\
& \le & 2^{nd} \int_{|t|\le 3\cdot 2^{-n}} r_K(t) dt
\end{eqnarray*}
if $\lambda, \lambda'\in 2^{-n}\Zd$ with $|\lambda-\lambda'|\le 2^{-n+1}$, and
\begin{eqnarray*}
|a_n(\lambda, \lambda')| &\le  &   2^{2nd}
\int_{|x-\lambda|\le 2^{-n-1}, |y-\lambda'|\le 2^{-n-1}}  r_{K} ((\lambda-\lambda')/2) dy dx
\\
&\le &    r_{K} ((\lambda-\lambda')/2)
\end{eqnarray*}
for all $\lambda, \lambda'\in 2^{-n}\Zd$ with  $|\lambda-\lambda'|> 2^{-n+1}$.
This proves \eqref{offdiagonaldecay.prop.eq1}.
\end{proof}

\subsection{Off-diagonal decay  of  commutators between discretization matrices and localization matrices}

\begin{prop}\label{liodiscretization.prop3}
 Let $1\le p<\infty$, $n\in \ZZ_+, N\in \NN$, $w$ be an $A_p$-weight,  $K$ be a kernel function on $\Rd$ satisfying \eqref{rkintegrable.eq}, 
 and let discretization matrices $A_n$, weights $w_n$,  and localization matrices
 $\Psi_k^N$ be as in \eqref{an.def}, \eqref{wnlambda.def} and \eqref{lio.eq6} respectively.  
Then there exists an absolute constant $C$,  depending on $p$ and $d$ only, such that for all $ b\in \ell_{w_n}^p$ and $k,k'\in N\Zd$,
\begin{eqnarray}\label{liodiscretization.prop3.eq2}
& & \|(\Psi_k^N A_n-A_n\Psi_k^N)\Psi_{k'}^N b\|_{p, w_n}\le C  (A_p(w))^{1/p} 2^{nd}  \|b\|_{p, w_n}
\nonumber\\
& &  \quad \times
 \begin{cases}
N^d r_K\big(\frac{k-k'}{2}\big) \Big(\frac{\sum_{|\lambda-k|\le 2N} w_n(\lambda)}{\sum_{|\lambda'-k'|\le 2N} w_n(\lambda')}\Big)^{1/p} & {\rm if} \ |k-k'|>8N,
\\
 \big(N^{-1/2} \|r_K\|_1+ \int_{|t|>\sqrt{N}/4 } r_K(t) dt\big)
 & {\rm if} \ |k-k'|\le 8N.
\end{cases}
\end{eqnarray}
\end{prop}

 A positive sequence $w=(w(k))_{k\in \Zd}$ is said to be a
{\em discrete $A_p$-weight} if for  all $a\in \Zd$  and    $N\in \NN$,
\begin{equation}\label{discreteaqweight.eq1}
\Big( N^{-d} \sum_{k\in a+[0,N-1]^d}  w(k)\Big)
\Big(N^{-d}\sum_{k\in a+[0,N-1]^d}  (w(k))^{-\frac{1}{p-1}} \Big)^{p-1}\le A<\infty
\end{equation}
when $1<p<\infty$, and
\begin{equation}\label{discreteaqweight.eq2}
 N^{-d} \sum_{k\in a+[0,N-1]^d}  w(k)
\le A \inf_{k\in a+[0,N-1]^d}  w(k)
\end{equation}
when $p=1$.
 The  smallest constant $A$ for which \eqref{discreteaqweight.eq1} holds when $1<p<\infty$ (for which \eqref{discreteaqweight.eq2} holds when $p=1$ respectively) is the {\em discrete $A_p$-bound}. We denote by $A_p(w)$ the discrete $A_p$-bound of a discrete $A_p$-weight $w$.
  To prove Proposition \ref{liodiscretization.prop3}, we  recall the boundedness of an infinite matrix on a weighted sequence space.

\begin{lem}\label{imboundedness.lem} {\rm (\cite[Theorem 3.2]{sunconst10})}\
Let $1\le p<\infty$, $w=(w(k))_{k\in \Zd}$ be a discrete $A_p$-weight, and
$A:=(a(k,k'))_{k,k'\in \Zd}$ be an infinite matrix with
$\|A\|_{{\mathcal B}}:=\sum_{m\in \Zd} (\sup_{|k-k'|\ge |m|} |a(k,k')|)<\infty$.
Then there exists an absolute constant $C$ (depending on $p$ and $d$ only) such that
$\|Ac\|_{p,w}\le C (A_p(w))^{1/p}\|A\|_{{\mathcal B}}\|c\|_{p,w}$ for all $c\in \ell^p_w$.
\end{lem}

\begin{proof} [Proof of Proposition \ref{liodiscretization.prop3}]
Write
$(\Psi_k^NA_n-A_n\Psi_k^N) \Psi_{k'}^N= (c(\lambda, \lambda'))_{\lambda, \lambda'\in 2^{-n}\Zd}$.
Then
 for  $|k-k'|\le 8N$,
\begin{eqnarray*} 
|c(\lambda, \lambda')|
& = &
 \Big|\Big(\psi_0\big(\frac{\lambda-k}{N}\big)-\psi_0\big(\frac{\lambda'-k}{N}\big)\Big) a_n(\lambda, \lambda')
\psi_0\big(\frac{\lambda'-k'}{N}\big)\Big|
\nonumber\\
& \le &  \min\big(\frac{|\lambda-\lambda'|}{N}, 1\big) |a_n(\lambda, \lambda')|
\psi_0\big(\frac{\lambda'-k'}{N}\big)
\nonumber\\
&\le &
\left\{\begin{array} {ll}
 2^{n(d-1)+1} N^{-1} \int_{|t|\le 3\cdot 2^{-n}} r_K(t) dt  &\  {\rm if} \ |\lambda-\lambda'|\le 2^{-n+1}\\
 \min(|\lambda-\lambda'|/N,1) r_K((\lambda-\lambda')/2)  & \   {\rm if} \ |\lambda-\lambda'|>2^{-n+1}
 \end{array}\right.
 \end{eqnarray*}
 by the Lipschitz property for  the function $\psi_0$ and  the off-diagonal property for the matrix $A_n$ in  Proposition \ref{offdiagonaldecay.prop}.
 Therefore 
 \begin{equation}\label{liodiscretization.prop3.pf.eq3-}
 |c(\lambda, \lambda')|\le C g(\lambda-\lambda')\quad {\rm for \ all}  \ \lambda, \lambda'\in 2^{-n}\Zd,
 \end{equation}
 where  $(g(\lambda))_{\lambda\in 2^{-n}\Zd}$  is  a radially decreasing sequence defined by
 \begin{eqnarray*}
g(\lambda) & = &
 \Big(\frac{2^{n(d-1)}}{N}\int_{|t|\le 3\cdot 2^{-n}}r_K(t) dt+
 {b_K(2^{-n})\over \sqrt{N}}+b_K\big({\sqrt{N}\over 2}\big)\Big)\chi_{[-2^{-n+1}, 2^{-n+1}]^d}(\lambda) \nonumber \\
&&+\Big({1\over \sqrt{N}}r_K\big({\lambda \over 2}\big)+b_K\big({\sqrt{N}\over 2}\big)\Big)
\big(\chi_{[-\sqrt{N}, \sqrt{N}]^d\backslash [2^{-n+1}, 2^{-n+1}]^d}(\lambda)\big)
\nonumber \\
&&+r_K\big({\lambda\over 2}\big)\big(1-\chi_{[-\sqrt{N},\sqrt{N}]^d}(\lambda)\big).
  \end{eqnarray*}
Note that
 \begin{eqnarray}\label{liodiscretization.prop3.pf.eq3}
\sum_{\lambda\in 2^{-n} {\mathbb Z}^d} g(\lambda)
& \le &  C\Big(2^{nd}  N^{-1/2}  \int_{|t|\le \sqrt{N}/2 } r_K(t) dt
+{b_K(2^{-n})\over\sqrt{N}}
\nonumber\\
& & \qquad +2^{nd}N^{d/2}b_K\Big({\sqrt{N}\over 2}\Big)+
2^{nd}   \int_{|t|>\sqrt{N}/4 } r_K(t) dt\Big)\nonumber\\
& \le & C 2^{nd} \Big(N^{-1/2} \|r_K\|_1+ \int_{|t|>\sqrt{N}/4 } r_K(t) dt\Big).
\end{eqnarray}
Then  the conclusion \eqref{liodiscretization.prop3.eq2} for $|k-k'|\le 8N$ follows from
\eqref{liodiscretization.prop3.pf.eq3-}, \eqref{liodiscretization.prop3.pf.eq3},  Lemma \ref{imboundedness.lem} and
Proposition \ref{discreteapweight.prop1}.

\smallskip

For $|k-k'|>8N$,
\begin{eqnarray}\label{liodiscretization.prop3.pf.eq4}
|c(\lambda, \lambda')|  & = &
 \big|\psi_0\big(\frac{\lambda-k}{N}\big) a_n(\lambda, \lambda')
\psi_0\big(\frac{\lambda'-k'}{N}\big)\big|\nonumber\\
&\le &  r_K((k-k')/2) \chi_{k+[-2N, 2N]^d}(\lambda)\chi_{k'+[-2N, 2N]^d}(\lambda')
\end{eqnarray}
by  Proposition \ref{offdiagonaldecay.prop}.
Write $b=(b(\lambda))_{\lambda\in 2^{-n}\Zd}$. Then by \eqref{discreteaqweight.eq1}  and
\eqref{liodiscretization.prop3.pf.eq4} 
we obtain that
  \begin{eqnarray*}
 & & \|(\Psi_k^NA_n-A_n\Psi_k^N)\Psi^N_{k'} b\|_{p, w_n}
\nonumber\\
 & \le &  r_K ((k-k')/2)
\Big(\sum_{|\lambda-k|\le 2N}   w_n(\lambda)\Big)^{1/p}
\nonumber\\
& &\quad  \times \Big(\sum_{|\lambda'-k'|\le 2N} |b(\lambda')|^p w_n(\lambda')\Big)^{1/p}
 \Big(\sum_{|\lambda'-k'|\le 2N} (w_n(\lambda'))^{-1/(p-1)}\Big)^{(p-1)/p}
 \nonumber\\
& \le &  2^{nd} N^d r_K((k-k')/2) (A_p(w))^{1/p}
 \Big(\frac{\sum_{|\lambda-k|\le 2N}   w_n(\lambda)}{\sum_{|\lambda'-k'|\le 2N}   w_n(\lambda')} \Big)^{1/p}  \|b\|_{p, w_n}
\end{eqnarray*}
for $1<p<\infty$, and similarly
\begin{eqnarray*}
& & \|(\Psi_k^NA_n-A_n\Psi_k^N)\Psi^N_{k'} c\|_{1, w_n}
\nonumber\\
  & \le &  2^{nd} N^d
   r_K((k-k')/2)  A_1(w)
 \Big(\frac{\sum_{|\lambda-k|\le 2N}   w_n(\lambda)}{\sum_{|\lambda'-k'|\le 2N}   w_n(\lambda')} \Big) \|b\|_{1, w_n}
   \end{eqnarray*}
for $p=1$.
Hence the conclusion \eqref{liodiscretization.prop3.eq2} for $|k-k'|>8N$ follows.
\end{proof}

\section{Stability of localized integral operators} \label{liostability.section}

To prove Theorem \ref{liostability.tm}, we need several technical lemmas.

\begin{lem}\label{liostability.lem1} Let $1\le p<\infty$, $z\in \CC$, $w$ be an $A_p$-weight, and let  the kernel $K$ and the integral operator $T$  with kernel $K$
be as in Theorem \ref{liostability.tm}. Set
\begin{equation}\label{delta0.def} \delta_0=\min(r_0/(2 A_p(w)), \alpha/(3d))\end{equation} where $\alpha\in (0,1]$ and  $r_0\in (0, 1)$ are given  in  \eqref{kernelcondition} and
Proposition \ref{apweight.prop3} respectively.
If $z I-T$ has $L^{p}_{w^r}$-stability for some $r\in (0,1]$,
then it has $L^p_{w^{r(1+s)}}$-stability for all $s\in [-\delta_0, \delta_0]$ with $0\le r(1+s)\le 1$.
\end{lem}

\begin{lem}\label{liostability.lem2}
Let $1\le p<\infty$, $z\in \CC$, $w$ be an $A_p$-weight, and let the kernel $K$ and the integral operator $T$  with kernel $K$
be as in Theorem \ref{liostability.tm}. Set
\begin{equation}\label{delta1.def} \delta_1=\min\big((p\ln 2+2\ln A_p(w))^{-1}D_1, ( 2(2^d+1)+2d +4 (2^d+1) \ln A_p(w))^{-1}\alpha \big)\end{equation} where $\alpha\in (0,1]$ and
$D_1\in (0, 1)$ are given  in  \eqref{kernelcondition} and
Proposition \ref{apweight.prop2} respectively.
If $zI-T$ has $L^{p}_{w^r}$-stability for some $r\in [0, \delta_1]$,
then it has $L^p_{w^{r'}}$-stability for all $r'\in [0, \delta_1]$.
\end{lem}

\begin{lem}\label{liostability.lem3}
Let $1\le p<\infty$, $z\in \CC$,  and let  the kernel $K$ and the integral operator $T$  with kernel $K$
be as in Theorem \ref{liostability.tm}.
Set
$\delta_2=\alpha/(3d)$ with $\alpha\in (0,1]$ given  in  \eqref{kernelcondition}. If $zI-T$ has $L^{p}$-stability,
then it has $L^{p(1+s)}$-stability for all $s\in [-\delta_2, \delta_2]$ with $p(1+s)\ge 1$.
\end{lem}

We assume that the conclusions in the  above three lemmas hold and  proceed to prove Theorem \ref{liostability.tm}  by  the bootstrap technique.

\begin{proof}[Proof of Theorem \ref{liostability.tm}]
We start from assuming that $zI-T$ has the $L^p_w$-stability for some $z\in \CC, p\in [1, \infty)$ and   $w\in {\mathcal A}_p$, and we want to prove that
$zI-T$ has the $L^{p'}_{w'}$-stability for any $ p'\in [1, \infty)$ and   $w'\in {\mathcal A}_{p'}$. 
Let $\delta_0$ and $\delta_1$ be as in \eqref{delta0.def} and \eqref{delta1.def} respectively, and select an integer $l_0$ sufficiently large such that $(1-\delta_0)^{l_0}\le \delta_1$.
 Iteratively applying  Lemma \ref{liostability.lem1} with
  $s=-\delta_0$  and $r=(1-\delta_0)^l$ for $l=0, 1, \ldots, l_0-1$,
we obtain that $zI-T$ has $L^p_{w^{(1-\delta_0)^l}}$-stability for all $l=1, \ldots, l_0$. Then applying  Lemma \ref{liostability.lem2} with $r=(1-\delta_0)^{l_0}$ and $r'=0$ leads to
the $L^p$-stability of $zI-T$.

 Select an integer $\ell_1\in \NN$
 and $s\in [-\delta_2, \delta_2]$ such that $(1+s)^{l_1}=p'/p$.
 Then  iteratively applying  Lemma \ref{liostability.lem3} with $p$ replaced by $p(1+s)^l, l=0, 1, \ldots, l_1-1$,
 yields the $L^{p'}$-stability of $zI-T$.

 Let  $\delta_0^\prime$ and $\delta_1^\prime$ be  as in Lemmas \ref{liostability.lem1} and
  \ref{liostability.lem2} with $p$ replaced by $p'$ and $w$ by $w'$, and select an integer $l_3\in \NN$ such that $(1+\delta_0^\prime)^{-l_3}\le \delta_1^\prime$. Applying  Lemma \ref{liostability.lem2} with $p$ replaced by $p'$, $w$ by $w'$, $r$ by $0$ and $r'$ by $(1+\delta_0^\prime)^{-l_3}$ leads to the  $L^{p'}_{(w')^{(1+\delta_0^\prime)^{-l_3}}}$-stability of $zI-T$. We then
  reach the desired  $L^{p'}_{w'}$-stability of the operator $zI-T$ by
   iteratively applying  Lemma \ref{liostability.lem1} with
   $p$ replaced by $p'$, $w$ by $w'$, $s$ by $\delta_0^\prime$
   and $r$ by $(1+\delta_0^\prime)^{-l_3+l}, l=0, 1, \cdots, l_3-1$.
 \end{proof}

\subsection{Proof of Lemma \ref{liostability.lem1}} Let $z I-T$ have the  $L^p_{w^r}$-stability. Then
there exists a positive constant $C_1$ such that
\begin{equation}\label{liostablity.lem1.eq1}
\|(z I-T)f\|_{p,w^r}\ge C_1\|f\|_{p,w^r} \quad {\rm for \ all} \ f\in L^p_{w^r}.
\end{equation}
From  Proposition \ref{lioapproximation.prop2} it follows that
\begin{eqnarray}\label{liostability.lem1.eq2}
\|(T-P_nTP_n)f\|_{p,w^r} & \le &  C_2 D_0 2^{-\alpha n} (A_p(w^r))^{1/p} \|f\|_{p, w^r}\nonumber\\
& \le &   C_2 D_0 2^{-\alpha n} (A_p(w))^{1/p} \|f\|_{p, w^r} \quad {\rm for \ all} \ f\in L^p_{w^r},
\end{eqnarray}
where $D_0=\|r_K\|_1+\sup_{0<\delta\le 1} \delta^{-\alpha} \|r_K \chi_{[-\delta, \delta]}\|_1+ \sup_{0<\delta\le 1}
\delta^{-\alpha}\|r_{\omega_\delta(K)}\|_1$ and $C_2$ is an absolute constant  in Proposition \ref{lioapproximation.prop2}. 
Let $n_0$ be a positive  integer such that
$ C_2 D_0 2^{-\alpha n_0} (A_p(w))^{1/p}\le C_1/2$.  
 Then
 for  all $n\ge n_0$ and $f\in L^p_{w^r}$,
\begin{equation}\label{liostability.lem1.eq3}
\|( z I-P_nTP_n)f\|_{p, w^r}\ge \frac{C_1}{2} \|f\|_{p, w^r}
\end{equation}
 by  \eqref{liostablity.lem1.eq1} and \eqref{liostability.lem1.eq2}. 
 Define
\begin{equation} \label{liostability.lem1.eq4}
(w^r)_n= \Big(2^{nd} \int_{\lambda+2^{-n}[-1/2,1/2)^d} (w(x))^r dx\Big)_{\lambda\in 2^{-n}\Zd}
\end{equation}
and
\begin{equation}\label{liostability.lem1.eq5}
(V^r)_n=\Big\{\sum_{ \lambda\in 2^{-n}\Zd} c(\lambda) \phi_{n, 2^n\lambda}\Big| \sum_{\lambda\in 2^{-n}\Zd} |c(\lambda)|^p (w^r)_n(\lambda)<\infty\Big\}.\end{equation}
Note that
for any $f_n:=\sum_{ \lambda\in 2^{-n}\Zd} c(\lambda) \phi_{n, 2^n\lambda}\in (V^r)_n$,
\begin{eqnarray} \label{liostability.lem1.eq6}
\|f_n\|_{p, w^r}   & =& \Big( 2^{ndp/2} \sum_{\lambda\in 2^{-n}\Zd} |c(\lambda)|^p \int_{\lambda+2^{-n}[-1/2,1/2)^d} w(x)^r dx\Big)^{1/p}\nonumber\\
& = &  2^{nd(1/2-1/p)} \|c\|_{p, (w^r)_n}
\end{eqnarray}
and
\begin{equation}\label{liostability.lem1.eq7}
\|(z I-P_nTP_n)f_n\|_{p, w^r}= 2^{nd(1/2-1/p)} \| (z I- 2^{-nd} A_n)c\|_{p, (w^r)_n}
\end{equation}
by Proposition \ref{liodiscretization.prop1},
where $A_n$ is defined in \eqref{an.def}.
Then applying \eqref{liostability.lem1.eq3} to $f_n\in (V^r)_n$, and using \eqref{liostability.lem1.eq6} and \eqref{liostability.lem1.eq7}, we obtain
a discretized version of the $L^p_{w^r}$-stability of  $zI-T$:
\begin{equation} \label{liostability.lem1.eq8}
\|(z I-2^{-nd} A_n) c\|_{p, (w^r)_n}\ge \frac{C_1}{2} \|c\|_{p, (w^r)_n} \quad {\rm for \ all}
\ c\in \ell^p_{(w^r)_n}\ {\rm and} \ n\ge n_0.
\end{equation}
To prove the $L^p_{w^{r(1+s)}}$-stability of  $zI-T$, we need the following claim, a weak version of the above stability with weight $w^r$ replaced by $w^{r(1+s)}$.

{\bf Claim 1}: {\em There exists a positive constant $\tilde C $ 
such that
\begin{equation}\label{liostability.lem1.eq888}
 \|(zI-2^{-nd}A_n)c\|_{p, (w^{r(1+s)})_n} \ge \tilde C 2^{-2nd|s|} \|c\|_{p, (w^{r(1+s)})_n}
 \end{equation}
 for all $c\in \ell^p_{w^{r(1+s)}_n}$ and $n\ge n_0$.
}

We assume that Claim 1 holds and proceed our proof.
Applying \eqref{liostability.lem1.eq6} and \eqref{liostability.lem1.eq7} with $f_n$ replaced by  $P_n f$ and $w^r$ by $w^{r(1+s)}$ and
using  \eqref{liostability.lem1.eq888}, we have
 \begin{equation}\label{liostability.lem1.eq21}
C_2 2^{-2nd|s|}\|P_nf\|_{p, w^{r(1+s)}}\le  \|(zI-P_nTP_n) P_n f\|_{p, w^{r(1+s)}}
\end{equation}
for all $f\in L^p_{w^{r(1+s)}}$ and  $n\ge n_0$.
As noted in  Remark \ref{lioapproximation.prop1},
\begin{eqnarray}\label{liostability.lem1.eq22}
\|g\|_{p, w^{r(1+s)}}   & \le &   \|P_ng\|_{p, w^{r(1+s)}}+\|(I-P_n)g\|_{p, w^{r(1+s)}}\nonumber\\
&  \le &    (1+2 A_p(w)) \|g\|_{p, w^{r(1+s)}}\quad {\rm for \ all} \ g\in L^p_{w^{r(1+s)}}.
\end{eqnarray}
Let integer $n_1$ be so chosen  that
 $\tilde C 2^{-2 n_1d\delta_0}\le |z|$ and  $C D_0 (A_p(w))^{1/p} 2^{-n_1\alpha/3}\le \tilde C/2$ where
 $C$ is the positive constant in Proposition \ref{lioboundedness.prop1}.
Recall that $\delta_0<\alpha/(3d)$ by assumption and $z\ne 0$ by \eqref{zerospectrum.cor.pf.eq1} and \eqref{liostablity.lem1.eq1}.
Then applying   \eqref{liostability.lem1.eq21} and \eqref{liostability.lem1.eq22} and letting $n=\max(n_0, n_1)$,
we obtain that
\begin{eqnarray}\label{discretetocontinuouspf}
& & \|(zI-T)f\|_{p, w^{r(1+s)}}
\nonumber\\
& \ge & (1+2 A_p(w))^{-1} \big( \|P_n(zI-T)f\|_{p, w^{r(1+s)}}+
\|(I-P_n)(zI-T)f\|_{p, w^{r(1+s)}}\big)
\nonumber\\
& \ge & (1+2 A_p(w))^{-1} \big( \|P_n(zI-T)P_nf\|_{p, w^{r(1+s)}}+ |z| \|(I-P_n)f\|_{p, w^{r(1+s)}}
\nonumber\\
& & -
\|P_n(zI-T)(I-P_n)f\|_{p, w^{r(1+s)}}- \|(I-P_n)Tf\|_{p, w^{r(1+s)}}\big)
\nonumber\\
& \ge & (1+2 A_p(w))^{-1}\big( \tilde C 2^{-2 nd|s|} \|P_nf\|_{p, w^{r(1+s)}}+ |z| \|(I-P_n)f\|_{p, w^{r(1+s)}}
\nonumber\\
& & -
C D_0 (A_p(w))^{1/p} 2^{-n\alpha} \|f\|_{p, w^{r(1+s)}}\big)
\nonumber\\
\quad & \ge &  (1+2 A_p(w))^{-1}\tilde  C 2^{-2 nd\delta_0-1}  \|f\|_{p, w^{r(1+s)}}
\end{eqnarray}
for all $f\in L^p_{w^{r(1+s)}}$ with $s\in [-\delta_0, \delta_0]$,
where $C$ is the positive constant in Proposition \ref{lioboundedness.prop1}.
This establishes the desired $L^p_{w^{r(1+s)}}$-stability for the operator $zI-T$ when $|s|\le \delta_0$.

\bigskip
Now it remains to prove  Claim 1.  Let $N$ be a sufficiently large integer chosen later and $\Psi_k^N, k\in N\Zd,$ be given in \eqref{lio.eq6}. Define
$\Phi_N=\big(\sum_{k\in N\Zd} (\Psi_k^N)^2\big)^{-1}$. 
Then $\Phi_N$ is a diagonal matrix with diagonal entries being positive and less than one, which implies that
\begin{equation}\label{liostability.lem1.eq9-1}
\|\Phi_N c\|_{p, (w^r)_n}\le \|c\|_{p, (w^r)_n}\quad {\rm for \ all} \ c\in \ell^p_{(w^r)_n}.
\end{equation}
Define
\begin{equation} \label{liostability.lem1.eq10}
(\alpha^r)_k=\sum_{|\lambda-k|\le 2N}  (w^r)_n(\lambda)=
 2^{nd}\int_{k+[-2N-2^{-n-1}, 2N+2^{-n-1})^d} w(x)^r dx, \ k\in N\Zd.
\end{equation}
By
\eqref{liostability.lem1.eq8}, 
 \eqref{liostability.lem1.eq9-1},
\eqref{liostability.lem1.eq10} and Proposition \ref{liodiscretization.prop3},  we get
\begin{eqnarray*}
& & \frac{C_1}{2} \frac{\|\Psi^N_kc\|_{p, (w^r)_n}}{ ((\alpha^r)_k)^{1/p} }
\le  \frac{\|(z I-2^{-n}A_n)\Psi^N_k  c\|_{q, (w^r)_n}}{ ((\alpha^r)_k)^{1/p} }
\nonumber\\
& \le & \frac{\|\Psi^N_k (z I-2^{-nd}A_n) c\|_{p, (w^r)_n}}{ ((\alpha^r)_k)^{1/p} }
\nonumber\\
&& \quad + 2^{-nd} \sum_{k'\in N\Zd} \frac{ \|(\Psi^N_k A_n-A_n \Psi_k^N)\Psi_{k'}^N \Phi_N \Psi_{k'}^N c\|_{p, (w^r)_n}}{ ((\alpha^r)_k)^{1/p} }
 \nonumber\\
& \le & \frac{\|\Psi^N_k (z I-2^{-nd}A_n) c\|_{p, (w^r)_n}}{ ((\alpha^r)_k)^{1/p} }
+  C_3 (A_p(w^r))^{1/p} \nonumber\\
&& \quad \times \sum_{|k'-k|\le 8N\atop k'\in N\Zd} \Big( N^{-1/2}\|r_K\|_1+ \int_{|t|\ge \sqrt{N}/4} r_K(t) dt\Big)
\frac{\| \Phi_N \Psi_{k'}^N c\|_{p, (w^r)_n}}{ ((\alpha^r)_k)^{1/p} }\nonumber\\
 & & + C_3 (A_p(w^r))^{1/p}  N^d
\sum_{|k-k'|>8N\atop k'\in N\Zd}  r_K((k-k')/2)  \frac{\| \Phi_N \Psi_{k'}^N c\|_{q, (w^r)_n}}{ ((\alpha^r)_{k'})^{1/p} }
 \end{eqnarray*}
for any bounded sequence $c$, where $C_3$ is an absolute constant depending on $p$ and $d$ only.
 Thus
  \begin{eqnarray}\label{liostability.lem1.eq11}
 \frac{\|\Psi^N_kc\|_{p, (w^r)_n}}{ ((\alpha^r)_k)^{1/p} }
& \le &  C_4 \frac{\|\Psi^N_k (z I-2^{-nd}A_n) c\|_{p, (w^r)_n}}{  ((\alpha^r)_k)^{1/p} }\nonumber\\
& &  + C_4
(A_p(w))^{1/p} \sum_{k'\in N\Zd}  g_N(k-k')  \frac{\|\Psi^N_{k'}c\|_{p, (w^r)_n}}{ ((\alpha^r)_{k'})^{1/p} } \end{eqnarray}
for any bounded sequence $c$,
where  $C_4$ is an absolute constant depending on $p$ and $d$ only, and
the  sequence $(g_N(k))_{k\in N\Zd}$ is defined by
\begin{eqnarray*} g_N(k) & = &
 \Big( N^{-1/2}\|r_K\|_1+ \int_{|t|\ge \sqrt{N}/4} r_K(t) dt\Big)\chi_{[-8N, 8N]^d}(k)
\nonumber\\
&  & +  N^d
 r_K(k/2) \chi_{N\Zd\backslash [-8N, 8N]^d}(k), \ k\in N\Zd.
 \end{eqnarray*}

 Let
 ${\mathcal B}$ contain all sequences $a:=(a(k))_{k\in \Zd}$ with $\|a\|_{\mathcal B}:=\sum_{m\in \Zd} \sup_{|k|\ge |m|}|a(k)|<\infty$
 (\cite{beurling49}),
 and denote by $a*b$ the convolution of two summable sequences $a$ and $b$ on $\Zd$.
 Recall that
 there exists a positive constant $D$ such that  $\|a*b\|_{\mathcal B}\le D \|a\|_{\mathcal B}\|b\|_{\mathcal B}$ for all $a, b\in {\mathcal B}$ \cite{baskakov11, belinskii97, beurling49, sunconst10}. Then $({\mathcal B}, \|\cdot\|_{\mathcal B}/D)$ is a Banach algebra under convolution.
Note that $(g_N(Nk))_{k\in \Zd}$ is a radially decreasing sequence, we then have
$$\|(g_N(Nk))_{k\in \Zd}\|_{\mathcal B}=\sum_{k\in N\Zd} g_N(k)\le C_5 \big(N^{-1/2}\|r_K\|_1+ \int_{|t|\ge \sqrt{N}/4} r_K(t) dt\big)\to 0$$
 as $N\to \infty$, where $C_5$ is an absolute constant depending on $p$ and $d$.
Now we select a sufficiently large integer $N$ so that
$$ C_4 C_5 (A_p(w))^{1/p} \Big(N^{-1/2}\|r_K\|_1+ \int_{|t|\ge \sqrt{N}/4} r_K(t) dt\Big)<\frac{1}{2D}.$$
Applying
\eqref{liostability.lem1.eq11} iteratively  and using the Banach algebra property  for ${\mathcal B}$,
we obtain that
 \begin{equation}\label{liostability.lem1.eq14}
 \frac{\|\Psi^{N}_kc\|_{p, (w^r)_n}}{ ((\alpha^r)_k)^{1/p} }
 \le 
  C_4 \sum_{k'\in N\Zd} V(k-k')  \frac{\|\Psi^{N}_{k'} (z I-2^{-nd}A_n) c\|_{p, (w^r)_n}}{ ((\alpha^r)_{k'})^{1/p} }
  \end{equation}
hold  for all bounded sequence $c$,
 where
\begin{equation}  V(k)=\delta(k)+\sum_{l=1}^\infty (C_4
(A_p(w))^{1/p})^l \underbrace{g_{N}*\cdots*g_{N}}_{l \ {\rm times}} (k)
\end{equation}
and $\delta(0)=1$ and $\delta(k)=0$ for all nonzero integer $k\in N\Zd$.
One may verify that
 \begin{equation} \label{liostability.lem1.eq15}
 \sum_{m\in N\Zd} \sup_{|l|\ge |m|} V(l)<\infty.
 \end{equation}

Set
 $
 Q_\lambda=\lambda+2^{-n}[-1/2,1/2)^d, \lambda\in 2^{-n}\Zd$ and
 $
 L_k=k+[-2N-2^{-n-1}, 2N+2^{-n-1})^d, k\in N\Zd$.
Then applying \eqref{apweight.eq5} with replacing $Q$ by
$L_k$ and
 $f$ by the characteristic function on $Q_\lambda$ and $w$ by $w^r$,  we have
\begin{eqnarray} \label{liostability.lem1.eq17}
 1 &\ge  &  \frac{\int_{Q_\lambda} w(x)^r dx}
 {\int_{L_k} w(x)^r dx}
 \ge   (A_p(w))^{-1} 2^{-ndp}(4N+1)^{-dp}
\end{eqnarray}
 for all $\lambda\in 2^{-n}\Zd$ and $k\in N\Zd$ with $|\lambda-k| \le 2N$.
For  $ k\in N\Zd$ and $c\in \ell^p_{w^{r(1+s)}_n}\cap \ell^\infty$, we obtain from
 \eqref{liostability.lem1.eq14}, \eqref{liostability.lem1.eq17}  and Proposition \ref{apweight.prop3} that
\begin{eqnarray}\label{liostability.lem1.eq18}
 & & \frac{\|\Psi^{N}_kc\|_{p, (w^{r(1+s)})_n}}{ ((\alpha^{r(1+s)})_k)^{1/p} }
 \nonumber\\
   & \le  & C(A_p(w))^{\frac{1+s}{p}} 2^{nds} \frac{\|\Psi^{N}_kc\|_{p, (w^{r})_n}}{ ((\alpha^{r})_k)^{1/p} }\nonumber\\
  & \le &   C(A_p(w))^{\frac{1+s}{p}} 2^{nds}
   \sum_{k'\in N\Zd}  V(k-k')  \frac{\|\Psi^{N}_{k'} (z I-2^{-nd}A_n)c\|_{p, (w^{r})_n}}{ ((\alpha^{r})_{k'})^{1/p} }
      \nonumber\\
  & \le & C  (A_p(w))^{3/p}
   2^{2nds} \sum_{k'\in N\Zd}  V(k-k')  \frac{\|\Psi^{N}_{k'} (z I-2^{-nd}A_n)c\|_{p, (w^{r(1+s)})_n}}{ ((\alpha^{r(1+s)})_{k'})^{1/p} }
\end{eqnarray}
for all $s\in [0, \delta_0]$, where $C$ is an absolute constant.
Similarly for all  $s\in [-\delta_0, 0]$ we have
\begin{eqnarray}\label{liostability.lem1.eq18+}
  \frac{\|\Psi^{N}_kc\|_{p, (w^{r(1+s)})_n}}{ ((\alpha^{r(1+s)})_k)^{1/p} }
  & \le  &  C  (A_p(w))^{3/p}
   2^{2nd |s|} \nonumber\\
   & & \times  \sum_{k'\in N\Zd}  V(k-k') \frac{\|\Psi^{N}_{k'} (zI-2^{-n}A_n)c\|_{p, (w^{r(1+s)})_n}}{ ((\alpha^{r(1+s)})_{k'})^{1/p} },
   \end{eqnarray}
where $ k\in N\Zd$ and $c\in \ell^p_{w^{r(1+s)}_n}\cap \ell^\infty$.
By Proposition \ref{discreteapweight.prop1} with $w$ replaced by $w^{r(1+s)}$,
$v_{N}= ((\alpha^{r(1+s)})_k)_{k\in N\Zd}$
is a discrete $A_p$-weight  with
$A_p(v_{N})\le  A_p(w^{r(1+s)})\le  A_p(w)$.
This, together with  \eqref{liostability.lem1.eq15}, \eqref{liostability.lem1.eq18}, \eqref{liostability.lem1.eq18+} and
 Lemma \ref{imboundedness.lem}, implies that
 \begin{eqnarray*}\label{liostability.lem1.eq20}
  \|c\|_{p, (w^{r(1+s)})_n}
& \le & \Big( \sum_{k\in N\Zd}\Big(  \frac{\|\Psi^{N}_kc\|_{p, (w^{r(1+s)})_n}}{ ((\alpha^{r(1+s)})_k)^{1/p} }\Big)^p
 (\alpha^{r(1+s)})_k\Big)^{1/p}\nonumber\\
  & \le  &  C_6 2^{2nd|s|}  \|(zI-2^{-nd}A_n)c\|_{p, (w^{r(1+s)})_n}
 \end{eqnarray*}
 for  all $c\in \ell^p_{(w^{r(1+s)})_n}\cap \ell^\infty$ and $n\ge n_0$,
 where  $C_6$ is an absolute constant independent of $n\ge n_0$ and $r\in (0,1]$ and $s\in [-\delta_0, \delta_0]$.
 Then  Claim 1 follows and Lemma \ref{liostability.lem1} is proved.

\subsection{Proof of Lemma \ref{liostability.lem2}}
 Let $zI-T$ have the $L^p_{w^r}$-stability. From  the argument used in the proof of Lemma
\ref{liostability.lem1},
there exist a sufficiently large integer $N$ and a sequence $V$  satisfying \eqref{liostability.lem1.eq15}
such that
\begin{equation}\label{liostability.lem2.pf.eq1}
 \frac{\|\Psi^{N}_kc\|_{p, (w^r)_n}}{ ((\alpha^r)_k)^{1/p} }
 \le 
  C_3 \sum_{k'\in N\Zd} V(k-k')  \frac{\|\Psi^{N}_{k'} (z I-2^{-nd}A_n) c\|_{p, (w^r)_n}}{ ((\alpha^r)_{k'})^{1/p} }
  \end{equation}
hold  for all bounded sequence $c$ and $k\in N\Zd$, where  $\Psi_k^N$ and $(\alpha^r)_k, k\in N\Zd$ are given in
\eqref{lio.eq6} and \eqref{liostability.lem1.eq10} respectively.
Note that $ L_k\subseteq 2^{n+5}   N Q_\lambda$ and $ 2^{n+1} NQ_\lambda\subset 2L_k$ when  $k\in N\Zd$ and $\lambda\in 2^{-n}\Zd$ with $|\lambda-k|\le 2N$.
Then by Proposition \ref{apweight.prop2}
\begin{equation}\label{liostability.lem2.pf.eq2}
C_1 2^{-d(p-1)rn} \le  \frac { 2^{nd}  \int_{\lambda+2^{-n}[-1/2, 1/2]^d} w(x)^r dx }{
\int_{[-2N-2^{-n-1}, 2N+2^{-n-1}]} w(x)^r dx}\le C_2 (2^{p} (A_p(w))^{2})^{(2^d+1)rn} \quad
\end{equation}
for all $r\in [0, \delta_1], k\in N\Zd$ and $\lambda\in 2^{-n}\Zd$ with $|\lambda-k|\le 2N $, where $C_1$ and $C_2$ are absolute constants.
Therefore for $r'\in [0, \delta_1],  k\in N\Zd$ and $\lambda\in 2^{-n}\Zd$ with $|\lambda-k|\le 2N$, we get from \eqref{liostability.lem2.pf.eq1} and \eqref{liostability.lem2.pf.eq2} that
\begin{eqnarray*} 
 \frac{\|\Psi^N_kc\|_{p, (w^{r'})_n}}{ ((\alpha^{r'})_k)^{1/p} }  & \le &  C
 2^{-nd/p}  (2^{p} (A_p(w))^{2})^{(2^d+1)r'n/p} \|\Psi^N_k c\|_p\nonumber \\
 & \le & C    (2^{p} (A_p(w))^{2})^{(2^d+1)r'n/p} 2^{d(p-1)rn/p}
 \frac{\|\Psi^N_kc\|_{p, (w^{r})_n}}{ ((\alpha^{r})_k)^{1/p} }\nonumber \\
 &\le  &
 C    (2^{p} (A_p(w))^{2})^{(2^d+1)r'n/p} 2^{d(p-1)rn/p}
 \sum_{k'\in N\Zd} V(k-k')\nonumber\\
   & &\times \frac{\|\Psi^N_{k'} (zI-2^{-nd}A_n) c\|_{p, (w^{r})_n}}{ ((\alpha^{r})_{k'})^{1/p} }\nonumber\\
  &\le  &
 C    (2^{p} (A_p(w))^{2})^{(2^d+1)(r+r') n/p} 2^{d(p-1)(r+r')n/p}
 \sum_{k'\in N\Zd} V(k-k')\nonumber\\
   & &\times \frac{\|\Psi^N_{k'} (zI-2^{-nd}A_n) c\|_{p, (w^{r'})_n}}{ ((\alpha^{r'})_{k'})^{1/p} }.
\end{eqnarray*}
This together with \eqref{liostability.lem1.eq15} and
Lemma \ref{imboundedness.lem} implies that
\begin{equation} \label{liostability.lem2.pf.pf.eq3}
\|c\|_{p, (w^{r'}) _n}\le
 C    (2^{p} (A_p(w))^{2})^{(2^d+1)(r+r') n/p} 2^{d(p-1)(r+r')n/p}
 \| (zI-2^{-n}A_n)c\|_{p, (w^{r'}) _n}
 \end{equation}
 for all bounded sequences $c$ in $\ell^p_{(w^{r'}) _n}$.
Therefore the desired $L^p_{w^{r'}}$-stability for the operator $zI-T$ follows by using
 the argument to establish \eqref{discretetocontinuouspf}
with applying \eqref{liostability.lem2.pf.pf.eq3} instead of
\eqref{liostability.lem1.eq888}.

\subsection{Proof of Lemma \ref{liostability.lem3}}
Let $zI-T$ has the $L^p$-stability. Similar to the argument to establish  \eqref{liostability.lem1.eq14},
there exist a sufficiently large integer $N$ and  a  sequence $V=(V(k))_{k\in N\Zd}$
satisfying  \eqref{liostability.lem1.eq15}
such that
\begin{equation}\label{liostability.lem3.pf.eq4}
\|\Psi_k^N c\|_p\le C \sum_{k'\in N\Zd} V(k-k') \|\Psi_k^N (zI-2^{-n}A_n) c\|_p
\end{equation}
for all bounded sequence $c$. Note that for $1\le q_1,  q_2<\infty$,
\begin{eqnarray}\label{liostability.lem3.pf.eq5}
(2^{d(n+2)}N^d)^{-\max(1/q_2-1/q_1, 0)} \|\Psi_k^Nc\|_{q_2} &\le & \|\Psi_k^N c\|_{q_1}\nonumber\\
&\le &  (2^{n+2}N)^{\max(1/q_1-1/q_2,0)} \|\Psi_k^Nc\|_{q_2}.
\end{eqnarray}
Combining \eqref{liostability.lem3.pf.eq4} and \eqref{liostability.lem3.pf.eq5} leads to
\begin{equation*}
\|\Psi_k^N c\|_{p(1+s)} 
  \le  C 2^{2nd |s|} \sum_{k'\in N\Zd} V(k-k') \|\Psi_k^N (zI-2^{-nd}A_n)c\|_{p(1+s)}
\end{equation*}
for all  bounded sequences $c$ and $s\in [-\delta_2, \delta_2]$.
Hence
\begin{equation}\label{liostability.lem3.pf.eq6}
\|c\|_{p(1+s)} 
  \le  C 2^{2nd |s|} \|(zI-2^{-nd}A_n)c\|_{p(1+s)}
\end{equation}
for all $c\in \ell^{p(1+s)}$.
Therefore the desired $L^{p(1+s)}$-stability of the operator $zI-T$ follows by using
 the argument to establish \eqref{discretetocontinuouspf}
with applying \eqref{liostability.lem3.pf.eq6} instead of
\eqref{liostability.lem1.eq888}.

\begin{appendix}

\section{Doubling property and reverse H\"older inequality for Muckenhoupt Weights}

In this appendix, we  provide some refinements of doubling property and reverse H\"older inequality for Muckenhoupt $A_p$-weights.
 Those refinements are important for the validation of the bootstrap technique used in the proof of Theorem
 \ref{liostability.tm}.

\subsection{Doubling property of Muckenhoupt $A_p$-weights}
An alternative way of defining Muckenhoupt $A_p$-weights is
\begin{equation}\label{apweight.eq5}
\Big(\frac{1}{|Q|} \int_{Q} |f(x)| dx\Big)^p\le \frac{A}{\int_Q w(x) dx}\int_Q |f(x)|^p w(x)dx \end{equation}
for all locally integrable functions $f$ and cubes $Q\subset \Rd$. The smallest constant $A$ for which \eqref{apweight.eq5} holds is the same as the $A_p$-bound $A_p(w), 1\le p<\infty$.
Applying \eqref{apweight.eq5} with $Q$ replaced by $2^nQ$ and
 $f$ by the characteristic  function on $Q$ gives that $wdx$ (or  $w$ for short) is a doubling measure; i.e.,
\begin{equation}\label{apweight.eq6}
\frac{1}{|2^nQ|} \int_{2^nQ} w(x)dx\le   2^{nd(p-1)} A_p(w)\Big(\frac{1}{|Q|}\int_Q w(x) dx \Big)
\end{equation}
for all positive integers $n$ and cubes $Q$ \cite{fourieranalysisbook, rubiobook}.
In  this subsection, we consider the doubling measure property of weights $w^r$ with sufficiently small $r>0$.

\begin{prop}
\label{apweight.prop2}
Let $1\le p<\infty$ and $w$ be an $A_p$-weight.
Then there exist absolute constants $C_0$ and $D_1$ (that depend on $p$ and $d$ only) such that
\begin{equation}\label{apweight.prop2.eq1}
(A_p(w))^{-r} 2^{-rnd(p-1)}  \le
\frac{\frac{1}{|Q|} \int_{Q} (w(x))^r dx}{\frac{1}{|2^n Q|}
\int_{2^n Q} (w(x))^r dx}
  \le   C_0 \big(2^{p}(A_p(w))^{2}\big)^{ (2^d+1)  rn}
\end{equation}
for all  integers $n\in \NN$, cubes $Q$ and  numbers $r\in [0,  D_1/(p\ln 2+2\ln A_p(w))]$.
\end{prop}

We say that a locally integrable function $f$ has {\em bounded mean oscillation}, or BMO for short,
if $\|f\|_{\rm BMO}:=\sup_{{\rm cubes} \ Q}\frac{1}{|Q|} \int_{Q} \big|f(x)-\frac{1}{|Q|} \int_Q f(y) dy\big| \ dx <\infty$.
To prove Proposition \ref{apweight.prop2}, we  recall that $\ln w$ has  bounded mean oscillation whenever $w$ is an $A_p$-weight for some $1\le p<\infty$
\cite{fourieranalysisbook, steinbook93}.

\begin{lem}\label{apweight.bmoprop} Let $1\le p<\infty$ and $w\in {\mathcal A}_p$. Then
 $\ln w$ has bounded mean oscillation and
\begin{equation} \label{apweight.bmoprop.eq1}
\|\ln w\|_{\rm BMO} \le  p
\ln 2+ 2\ln  A_p(w).
\end{equation}
\end{lem}

\begin{proof}  We follow the arguments in \cite[p. 151]{fourieranalysisbook} and \cite[p.197]{steinbook93}, and include a proof
for the BMO bound estimate in \eqref{apweight.bmoprop.eq1}  that will be used for our establishment of  Proposition \ref{apweight.prop2}.
Let $w$ be an $A_p$-weight with $1<p<\infty$.
Take an arbitrary cube $Q\subset \Rd$ and denote by $c_Q:=\frac{1}{|Q|} \int_Q \ln w(y) dy$  the average of the function $\ln w$ on the cube $Q$.
As $w$ is an $A_p$-weight,
\begin{equation}\label{apweight.bmoprop.pf.eq1}
 \Big(\frac{1}{|Q|} \int_Q e^{\ln w(x)-c_Q} dx\Big)
\Big(\frac{1}{|Q|} \int_Q e^{-(\ln w(x)-c_Q)/(p-1)} dx\Big)^{p-1}\le A_p(w).
\end{equation}
Note that \begin{equation}\label{apweight.bmoprop.pf.eq3}
 \Big(\frac{1}{|Q|} \int_Q e^{\ln w(x)-c_Q} dx\Big)\ge 1\quad {\rm and}\quad \Big(\frac{1}{|Q|} \int_Q e^{-(\ln w(x)-c_Q)/(p-1)} dx\Big)\ge 1
%
\end{equation}
by applying  Jensen's inequality 
  \begin{equation}\label{apweight.bmoprop.pf.eq2} \exp\Big(\frac{1}{|Q|}\int_Q f(x) dx\Big)\le \frac{1}{|Q|} \int_{Q}  e^{f(x)} dx\end{equation}
  with $f$ replaced by
$(\ln w(x) -c_Q)$
and $-(\ln w(x) -c_Q)/(p-1)$ respectively.
 Thus combining \eqref{apweight.bmoprop.pf.eq1} and \eqref{apweight.bmoprop.pf.eq3}, 
we have
 \begin{equation} \label{apweight.bmoprop.pf.eq4}
 \frac{1}{|Q|} \int_Q e^{\ln w(x)-c_Q} dx\le A_p(w) \quad {\rm  and} \quad  \frac{1}{|Q|} \int_Q e^{-(\ln w(x)-c_Q)/(p-1)} dx\le (A_p(w))^{1/(p-1)}.\end{equation}
Using the estimates in \eqref{apweight.bmoprop.pf.eq4} and applying Jensen's inequality \eqref{apweight.bmoprop.pf.eq2} with $f$ replaced by
$\max( \ln w(x) -c_Q, 0)$ and $\max( c_Q-\ln w(x), 0)/(p-1)$ respectively, we get
\begin{eqnarray}\label{apweight.bmoprop.pf.eq5}
  & & \exp\Big(\frac{1}{|Q|}\int_{Q} \max( \ln w(x) -c_Q, 0) dx\Big)
   \le \frac{1}{|Q|}\int_Q  e^{\max(\ln w(x)-c_Q,0)} dx\nonumber\\
& \le &  \frac{1}{|Q|}\int_Q  e^{\ln w(x)-c_Q} dx+ \frac{1}{|Q|}\int_Q e^0 dx \le  A_p(w)+1\le 2 A_p(w)
\end{eqnarray}
and
\begin{eqnarray}\label{apweight.bmoprop.pf.eq6}
\quad  & & \exp\Big(\frac{1}{|Q|}\int_{Q}\frac{ \max(c_Q-\ln w(x), 0)}{p-1} dx \Big)
\le
\frac{1}{|Q|} \int_Q e^{ \max(c_Q-\ln w(x), 0)/(p-1)}dx\nonumber\\
\quad & \quad \le & \frac{1}{|Q|}\int_Q e^{(c_Q-\ln w(x))/(p-1)} dx+ \frac{1}{|Q|} \int_Q e^0 dx 
\le
2 (A_p(w))^{1/(p-1)}.
\end{eqnarray}
The desired BMO bound estimate \eqref{apweight.bmoprop.eq1}
then follows from \eqref{apweight.bmoprop.pf.eq5} and \eqref{apweight.bmoprop.pf.eq6}.

The desired conclusion \eqref{apweight.bmoprop.eq1} for $p=1$ follows from the established result for $1<p<\infty$ and the fact that any $A_1$-weight $w$ is an $A_p$-weight with $A_p(w)\le A_1(w)$ for all $1<p<\infty$.
\end{proof}

\begin{lem}\label{apweight.lem2} Let $1\le p<\infty$ and $w\in {\mathcal A}_p$. Then there exist absolute
positive constants $C$ and $D_1$ (that depend on $p$ and $d$ only) such that
\begin{equation}\label{apweight.lem2.eq1}
\exp\Big (\frac{r}{|Q|} \int_{Q} \ln w(x) dx\Big)\le
\frac{1}{|Q|} \int_{Q} (w(x))^r dx\le
C \exp \Big(\frac{r}{|Q|} \int_{Q} \ln w(x) dx\Big)
\end{equation}
hold for all cubes $Q$ and all $r\in [0, D_1/(p\ln 2+2\ln A_p(w))]$.
\end{lem}

\begin{proof}
The first inequality in \eqref{apweight.lem2.eq1} follows by applying   Jensen's inequality \eqref{apweight.bmoprop.pf.eq2}
  with $f$ replaced by $r\ln w$.

For $p=1$ and $0<r\le D_1/\big(p\ln 2+2\ln A_p(w)\big)$,
\begin{eqnarray*} \frac{1}{|Q|} \int_{Q} (w(x))^r dx & \le &  \Big(\frac{1}{|Q|} \int_{Q} w(x) dx\Big)^r\le (A_1(w))^r \inf_{x\in Q} ( w(x))^r\nonumber\\
& \le &
e^{D_1/2} \exp \Big(\frac{r}{|Q|} \int_{Q} \ln w(x) dx\Big) \quad {\rm for\ all\ cubes} \ Q, \end{eqnarray*}
which leads to the second inequality in  \eqref{apweight.lem2.eq1} for $p=1$.
Now we prove the second inequality in \eqref{apweight.lem2.eq1} provided that $1<p<\infty$.
By  Lemma \ref{apweight.bmoprop} and  the  John-Nirenberg inequality for  functions with bounded mean oscillation, there exist  absolute positive
 constants $D_1$ and $D_2$
 such that
\begin{eqnarray*} \label{apweight.prop1.pf.eq10}
 |\{x\in Q: |\ln w(x)-c_Q|>\alpha\}|  &  \le   &  D_2\exp(- 2D_1 \alpha/\|\ln w\|_{\rm BMO}) |Q|\nonumber\\
& \le &  D_2 \exp\Big(-\frac{ 2D_1 \alpha }{p\ln 2+2\ln A_p(w)} \Big) |Q|
\end{eqnarray*}
for all  cubes $Q$,  where $\alpha>0$ and  $c_Q:=\frac{1}{|Q|} \int_Q \ln w(y) dy$ is  the average of the function $\ln w$ on the cube $Q$.
Therefore
\begin{eqnarray*}\label{apweight.prop1.pf.eq11}
& & \frac{1}{|Q|} \int_{Q} e^{r | \ln w(x)-c_Q|} dx
  =  1+ \frac{1}{|Q|}\int_0^\infty e^{t} |\{x\in Q: |\ln w(x)-c_Q|>t/r\}| d t
 \nonumber\\
 & \le &
 1+D_2\int_0^\infty \exp\Big(t - t \frac{2D_1}{  r (p\ln 2+2\ln A_p(w))}\Big) dt
\le   1+D_2 
\end{eqnarray*}
for all $r\in [0,D_1/\big(p\ln 2+2\ln A_p(w)\big)]$.
Thus
\begin{equation*}\label{apweight.prop1.pf.eq12}
\frac{1}{|Q|} \int_Q w(x)^r dx\le
\frac{ e^{r c_Q}}{|Q|} \int_Q e^{r|\ln w(x)- c_Q|} dx
\le (1+D_2) \exp\Big(  \frac{r}{|Q|} \int_Q \ln w(x) dx\Big)
\end{equation*}
and the second inequality in \eqref{apweight.lem2.eq1}   for $1<p<\infty$ follows.
\end{proof}

Now we prove Proposition \ref{apweight.prop2}.

\begin{proof}[Proof of Proposition \ref{apweight.prop2}] Let $1\le p<\infty$ and $w$ be an $A_p$-weight. Then for $0<r\le 1$,
$w^r\in {\mathcal A}_{1+r(p-1)}$  with its $A_{1+r(p-1)}$-bound dominated by $(A_p(w))^r$.
Then
applying \eqref{apweight.eq6} with $w$ replaced by $w^r$ and $p$ by $1+r(p-1)$, we obtain
$$ 
 \frac{\frac{1}{|Q|} \int_{Q} (w(x))^r dx}{\frac{1}{|2^n Q|}
\int_{2^nQ} (w(x))^r dx}
  \ge   \big(A_{1+r(p-1)}(w^r)\big)^{-1} 2^{-rnd(p-1)}
\ge (A_q(w))^{-r} 2^{-rnd(p-1)}
$$
for all positive integer $n$ and cubes $Q$.  This establishes
 the first inequality in \eqref{apweight.prop2.eq1}.

 By Lemmas \ref{apweight.bmoprop} and  \ref{apweight.lem2}, we get
\begin{eqnarray*}
& &\frac{\frac{1}{|Q|} \int_{Q} (w(x))^r dx}{\frac{1}{|2^nQ|}
\int_{2^nQ} (w(x))^r dx}
\le   C \exp \Big(r  \Big|\frac{1}{|2^nQ|}\int_{2^nQ} \ln w(x) dx-
\frac{1}{|Q|} \int_Q \ln w(x) dx\Big|\Big)\nonumber\\
 & \le & C \exp \Big(r  \sum_{k=0}^{n-1} \Big|\frac{1}{|2^{k+1}Q|}\int_{2^{k+1}Q} \ln w(x) dx-
\frac{1}{|2^kQ|} \int_{2^kQ} \ln w(x) dx\Big|\Big)\nonumber\\
 &\le  & C \exp \big((2^d+1) r n \|\ln w\|_{\rm BMO}\big)\le
C \exp \big((2^d+1) r n \big(p\ln 2+2\ln A_p(w)\big)\big).
\end{eqnarray*}
This proves   the second inequality in \eqref{apweight.prop2.eq1}.
\end{proof}

\subsection{Reverse H\"older inequality for Muckenhoupt $A_p$-weights}  One of key results for Muckenhoupt $A_p$-weights is the reverse H\"older inequality,
which states that for any $A_p$-weight $w, 1\le p<\infty$, there exist constants $C$ and $\epsilon>0$ (depending on $p, d$ and $A_p(w)$ only) such that
$\big(\frac{1}{|Q|} \int_Q w(x)^{1+\epsilon} dx\big)^{1/(1+\epsilon)}\le \frac{C}{|Q|} \int_Q w(x) dx$
for any cube $Q$ \cite{fourieranalysisbook, steinbook93}. In this subsection, we consider the  reverse H\"older inequality for weights $w^r, r\in [0,1]$.

\begin{prop}
\label{apweight.prop3}
Let $1\le p<\infty$ and $w\in {\mathcal A}_p$. Then there exist  a positive constant  $r_0$
 (depending on $p$ and $d$ only) such that
\begin{eqnarray}\label{apweight.prop3.eq1}
\big (2^{p+2} A_p(w)\big)^{-1} \Big(\frac{1}{|Q|} \int_Q w(x)^{(1+\delta) r} dx\Big)^{1/(1+\delta)} &\!\!\le &\!\!
\frac{1}{|Q|} \int_Q w(x)^{r} dx\nonumber\\
  &\!\! \le &\!\!  2^{p+2} A_p(w) \Big(\frac{1}{|Q|} \int_Q w(x)^{(1-\delta) r} dx\Big)^{1/(1-\delta)}
 \end{eqnarray}
hold for all cubes $Q$ and  positive numbers $r\in (0,1]$ and $\delta\in (0, r_0/A_p(w)]$.
\end{prop}

\begin{proof} We follow the argument in \cite[pp. 202--203]{steinbook93}. Let $r\in (0,1]$ and $w\in {\mathcal A}_p$ for some $1\le p<\infty$.
Then $w^r\in {\mathcal A}_{1+r(p-1)}\subset {\mathcal A}_{p}$ and
$A_p(w^r)\le A_{1+r(p-1)}(w^r)\le (A_p(w))^r\le A_p(w)$.
Therefore taking the characteristic function on a subset $E$ of a cube $Q$
in (\ref{apweight.eq5}) 
leads to
\begin{equation*}\label{apweight.prop3.pf.eq2}
\frac{\int_E w(x)^r dx}{\int_Q w(x)^r dx}\ge \frac{1}{A_p(w)}\Big(\frac{|E|}{|Q|}\Big)^p
\end{equation*}
for any subset $E\subset Q$.
This implies that for all cubes $Q$ and  subsets $E\subset Q$ with $|E|\le |Q|/2$,
\begin{equation*}\label{apweight.prop3.pf.eq3}
\frac{\int_E w(x)^r dx}{\int_Q w(x)^r dx}\le 1-\frac{1}{A_p(w)}
\Big(\frac{|Q-E|}{|Q|}\Big)^p\le\frac{2^pA_p(w)-1}{2^p A_p(w)}.
\end{equation*}
Let $\delta_1= \big(2^{p+3} (d+1) A_p(w)\big)^{-1}$. Then
$2^{2(d+1)\delta_1}(1-\frac{1}{2^p A_p(w)})\le (1-\frac{1}{2^{p+1} A_p(w)})$
 and for any $\delta\in (0, \delta_1]$, following the steps in \cite[pp.202--203]{steinbook93} we get
\begin{eqnarray*} \label{apweight.prop3.pf.eq6}
  & & \Big(\frac{1}{|Q|} \int_{Q} w(x)^{r(1+\delta)}dx\Big)^{1/(1+\delta)}\le  \Big(1+\sum_{k=0}^\infty
2^{(d+1)(k+1)\delta}\Big(1-\frac{1}{2^p A_p(w)}\Big)^k\Big)^{1/(1+\delta)}
\nonumber\\
 &   &
\qquad \quad\times \Big(\frac{1}{|Q|} \int_{Q} w(x)^{r}dx\Big)
\le 2^{p+2} A_p(w) \Big(\frac{1}{|Q|} \int_{Q} w(x)^{r}dx\Big) 
\end{eqnarray*}
and
\begin{eqnarray*} 
  & & \frac{1}{|Q|} \int_{Q} w(x)^{r}dx
  \le  \Big(1+\sum_{k=0}^\infty
2^{(d+1)(k+1)\delta/(1-\delta)}\Big(1-\frac{1}{2^p A_p(w)}\Big)^k\Big)\nonumber\\
& & \qquad  \times \Big(\frac{1}{|Q|} \int_{Q} w(x)^{r(1-\delta)}dx\Big)^{1/(1-\delta)}\le
2^{p+2} A_p(w) \Big(\frac{1}{|Q|} \int_{Q} w(x)^{r(1-\delta)}dx\Big)^{1/(1-\delta)}. 
\end{eqnarray*}
This establishes  \eqref{apweight.prop3.eq1} and completes the proof.
\end{proof}

\subsection{Discrete Muckenhoupt weights}
Muckenhoupt $A_p$-weights and discrete $A_p$-weights  are closely related.
Given a  discrete  $A_p$-weight $w=(w(k))_{k\in \Zd}$, one may verify that
$\tilde w:=\sum_{k\in \Zd} w(k) \chi_{[-1/2,1/2)^d}(\cdot-k)$
  is an  $A_p$-weight with its $A_p$-bound comparable to the $A_p$-bound of the discrete weight $w$.
   Conversely,  discretization of an $A_p$-weight at any level is a discrete $A_p$-weight.

\begin{prop}\label{discreteapweight.prop1}
Let $1\le p<\infty$ and  $w$ be an $A_p$-weight, and define
 \begin{equation*}\label{discreteapweight.eq1}
 w_n(k)=2^{nd} \int_{2^{-n}(k+[-1/2,1/2)^d)} w(x) dx,\quad n\in \ZZ, k\in \Zd.\end{equation*}
Then for any $n\in \ZZ$, $w_n:=(w_n(k))_{k\in \Zd}$ is  a discrete $A_p$-weight with its $A_p$-bound  dominated by the $A_p$-bound of the weight $w$, i.e.,
$A_p(w_n)\le A_p(w)$.
\end{prop}

\begin{proof}
 Let $1<p<\infty$ and $n\in \ZZ$. Given $a\in \Zd$ and $N\in \NN$,
  \begin{eqnarray*}
 & & \Big(\frac{1}{N^d} \sum_{k\in a+[0,N-1]^d} w_n(k)\Big)\Big(\frac{1}{N^d} \sum_{k\in a+[0,N-1]^d} (w_n(k))^{-1/(p-1)}\Big)^{p-1}\\
  & \le & \Big(\frac{1}{2^{-nd} N^d} \int_{2^{-n}a+2^{-n}[-1/2, N-1/2)^d} w(x)dx\Big)\\
  & &
  \times \Big(\frac{1}{2^{-nd}N^{d}}
  \int_{2^{-n}a+2^{-n}[-1/2, N-1/2)^d} w(x)^{-1/(p-1)} dx\Big)^{p-1} \le A_p(w)
  \end{eqnarray*}
   where the  first inequality follows from
 \begin{eqnarray*}
1 & \le & \Big(2^{nd}\int_{2^{-n}k+2^{-n}[-1/2, 1/2)^d} w(x) dx\Big)\\
& & \times \Big(2^{nd}\int_{2^{-n}k+2^{-n}[-1/2, 1/2)^d}  w(x)^{-1/(p-1)} dx\Big)^{p-1}\quad {\rm for \ all} \  k\in \Zd,
 \end{eqnarray*}
 and the second inequality holds as  $ |2^{-n}a+2^{-n}[-1/2, N-1/2)^d|=2^{-nd} N^d$.

 The conclusion for $p=1$ can be proved by similar argument.
\end{proof}

\end{appendix}


\begin{thebibliography}{999}




\bibitem{barnes90} B. A. Barnes, When is the spectrum of a
convolution operator on $L^p$ independent of $p$? {\em Proc.
Edinburgh Math. Soc.}, {\bf 33}(1990), 327--332.

\bibitem{baskakov11} A. G. Baskakov and I. Krishtal, Memory estimation of inverse operators, arXiv:1103.2748




\bibitem{belinskii97} E. S. Belinskii, E. R. Lifyand, and R. M. Trigub, The Banach algebra $A^*$ and its properties, {\em J. Fourier Anal. Appl}, {\bf 3}(1997), 103--129.

\bibitem{beurling49} A. Beurling, On the spectral synthesis of bounded functions, {\em Acta Math.},  {\bf 81}(1949),
 225–-238.




\bibitem{brandenburg75} L. Brandenburg, On identifying the maximal ideals in Banach algebra, {\em J. Math. Anal. Appl.}, {\bf 50}(1975), 489--510.




\bibitem{daubecheiesbook} I. Daubechies, {\em Ten Lectures on Wavelets}, CBMS-NSF Regional Conference Series in Applied Mathematics,
 SIAM, 1992.




\bibitem{fourieranalysisbook} J. Duoandikoetxea, {\em Fourier Analysis}, Amer. Math. Soc., 2000.

\bibitem{farrell10} B. Farrell  and T. Strohmer,
Inverse-closedness of a Banach algebra of integral operators on the Heisenberg group, {\em J. Operator Theory},
{\bf 64}(2010), 189--205.


\bibitem{rubiobook} J. Garcia-Cuerva and J.-L. Rubio De Francia,
{\em  Weighted Norm Inequalities and Related Topics}, Elsevier, 1985.





\bibitem{grochenigsurvey} K. Gr\"ochenig, Wiener's lemma: theme and variations, an introduction to
spectral invariance and its applications, In {\em Four Short Courses on Harmonic Analysis: Wavelets, Frames, Time-Frequency Methods, and Applications to Signal and Image Analysis}, edited by P. Massopust and B. Forster,
Birkhauser, Boston 2010.






\bibitem{hulanicki} A. Hulanicki, On the spectrum of convolution
operators on groups with polynomial growth, {\em Invent. Math.},
{\bf 17}(1972), 135--142.



\bibitem{kurbatovbook} V. G. Kurbatov, {\em Functional Differential Operators and Equations}, Kluwer Academic Publishers,
1999.


\bibitem{pytlik} T. Pytlik, On the spectral radius of elements in
group algrebras, {\em Bull. Acad. Polon. Sci. Ser. Sci. Math.},
{\bf 21}(1973), 899--902.

\bibitem{shincjfa09} C. E. Shin and Q. Sun,
 Stability of localized operators, {\em J. Funct. Anal.}, {\bf 256}(2009), 2417-2439.




\bibitem{steinbook93} E. M. Stein, {\em Harmonic Analysis: Real-Variable Methods, Orthogonality, and Oscillatory Integrals}, Princeton University Press, 1993.


\bibitem{steinbook70} E. M. Stein, {\em  Singular Integrals and Differentiability Properties of Functions},  Princeton University Press, 1970.



\bibitem{suntams07} Q. Sun, Wiener's lemma for infinite matrices,
{\em Trans. Amer. Math. Soc.},  {\bf 359}(2007), 3099--3123.

\bibitem{sunconst10}
Q. Sun, Wiener's lemma for infinite matrices II, {\em Constr. Approx.}, DOI: 10.1007/s00365-010-9121-8


\bibitem{sunacha08} Q.  Sun, Wiener's lemma for localized integral operators, {\em Appl. Comput. Harmonic Anal.}, {\bf 25}(2008), 148--167.



\end {thebibliography}
\end{document}